\documentclass[11pt]{amsart}
\usepackage[all]{xy}

\usepackage[utf8]{inputenc}		
\usepackage[T1]{fontenc}
\usepackage[english]{babel}

\usepackage{amsfonts}			
\usepackage{amsmath}
\usepackage{amssymb}
\usepackage{amsthm}
\usepackage{mathtools}
\makeatletter					
\@namedef{subjclassname@2020}{	
	\textup{2020} Mathematics Subject Classification}
\makeatother

\usepackage{graphicx}			
\usepackage{tikz}
	\usetikzlibrary{cd}

\usepackage{hyperref}			
	\hypersetup{colorlinks=false, urlcolor=black, linkcolor=black}

\usepackage{enumitem}			

\usepackage{color}				

\newcommand{\N}{\mathbb{N}}						
\newcommand{\R}{\mathbb{R}}						
\newcommand{\C}{\mathbb{C}}						

\renewcommand{\S}{\mathbb{S}}					

\newcommand{\eps}{\varepsilon}					

\newcommand{\dd}								
	{\mathop{}\!\mathrm{d}}						
\newcommand{\ddn}[1]							
	{\mathop{}\!\mathrm{d^{#1}}}

\newcommand{\abs}[1]							
	{\left| #1 \right|}
\newcommand{\smallabs}[1]						
	{\lvert #1 \rvert}	
\newcommand{\norm}[1]							
	{\left\lVert #1 \right\rVert}	
\newcommand{\smallnorm}[1]						
	{\lVert #1 \rVert}						
\newcommand{\ip}[2]								
	{\left< #1 , #2 \right>}

\DeclareMathOperator{\id}{id}					
\DeclareMathOperator{\vol}{vol}					
\DeclareMathOperator{\spt}{spt}
\DeclareMathOperator{\dist}{dist}					

\DeclareMathOperator{\osc}{osc}					

\DeclareMathOperator{\maxfun}{{\bf M}}			

\newcommand{\loc}{\mathrm{loc}}					

\newtheorem{thm}{Theorem}[section]{\bf}{\it}
\newtheorem{lemma}[thm]{Lemma}

\newtheorem{innercustomthm}{Theorem}[section]{\bf}{\it}
\newenvironment{customthm}[1]
	{\renewcommand\theinnercustomthm{#1}\innercustomthm}
	{\endinnercustomthm}
	
{\bf}{\it}

\theoremstyle{definition}
\newtheorem{defn}[thm]{Definition}
\newtheorem{ex}[thm]{Example}

\theoremstyle{remark}

\numberwithin{equation}{section}

\begin{document}
	
\title{On the heterogeneous distortion inequality}
\author{Ilmari Kangasniemi}
\address{Department of Mathematics, Syracuse University, Syracuse,
NY 13244, USA }
\email{kikangas@syr.edu}

\author{Jani Onninen}
\address{Department of Mathematics, Syracuse University, Syracuse,
NY 13244, USA and  Department of Mathematics and Statistics, P.O.Box 35 (MaD) FI-40014 University of Jyv\"askyl\"a, Finland
}
\email{jkonnine@syr.edu}
	
\thanks{J. Onninen was supported by the NSF grant  DMS-1700274.}
	
\begin{abstract}
	We study Sobolev mappings $f \in W_{\loc}^{1,n} (\R^n, \R^n)$, $n \ge 2$, that  satisfy the heterogeneous distortion inequality
	\[\abs{Df(x)}^n \leq K J_f(x) + \sigma^n(x) \abs{f(x)}^n\] 
	for almost every $x \in \R^n$. Here $K \in [1, \infty)$ is a constant and $\sigma \geq 0$ is a function in  $L^n_\loc(\R^n)$.  Although we recover the class of $K$-quasiregular mappings when $\sigma \equiv 0$, the theory of arbitrary solutions is significantly more complicated, partly due to the unavailability of a robust degree theory  for non-quasiregular solutions. Nonetheless, we obtain a Liouville-type theorem and the sharp H\"older continuity estimate for all solutions, provided that  $\sigma \in  L^{n-\eps}(\R^n) \cap L^{n+\eps}(\R^n)$ for some $\eps >0$.
	This gives an affirmative answer to a question of Astala, Iwaniec and Martin.
\end{abstract}
	
\maketitle
	
\section{Introduction}
Let $\Omega$ be a  connected, open subset of $\R^n$ with $n \ge 2$. We study mappings $f=(f_1, \dots , f_n) \colon \Omega \to \R^n$ in the Sobolev space $W_{\loc}^{1,n} (\Omega, \R^n)$  that satisfy the \emph{heterogeneous distortion inequality}
\begin{equation}\label{eq:main_definition}
	\abs{Df(x)}^n \leq K J_f(x) + \sigma^n(x) \abs{f(x)}^n
\end{equation}
for a.e.\ (almost every) $x \in \Omega$, where $K \in [1, \infty)$ and $\sigma \in L_{\loc}^n(\Omega)$. Here, $\abs{Df(x)}$ is the operator norm of the weak derivative $Df(x) \colon \R^n \to \R^n$ of $f$ at a point $x \in \Omega$; that is, $\abs{Df(x)} = \sup \{\abs{Df(x) h} : h \in \S^{n-1} \}$. Moreover, $J_f(x) = \det Df(x)$ is the Jacobian determinant of $f$. It is worth noting that the natural Sobolev space in which to seek the solutions is $W^{1,n}_{\loc} (\Omega, \R^n)$. The reason for this is that this space provides enough regularity to apply integration by parts to the form $J_f \dd x$.
	
If  $\sigma \equiv 0$ in~\eqref{eq:main_definition}, we recover the \emph{mappings of bounded distortion}, also
known as $K$-\emph{quasiregular} maps. Homeomorphic $K$-quasiregular maps are also commonly called \emph{$K$-quasiconformal}.  The theory of mappings of bounded distortion  arose from the need to generalize  the geometry of holomorphic functions to higher dimensions, and is  by now a central topic in modern analysis with important connections to partial differential equations, complex dynamics, differential geometry and the calculus of variations; see the monographs by Reshetnyak \cite{Reshetnyak-book}, Rickman \cite{Rickman_book}, Iwaniec and Martin \cite{Iwaniec-Martin_book}, and Astala, Iwaniec and Martin \cite{Astala-Iwaniec-Martin_Book}. The last 20 years have also seen widespread study of a more general class of deformations, \emph{mappings of finite distortion}, where the constant $K$ is replaced by a finite function $K \colon \Omega \to [1, \infty)$; see e.g.\ the monographs \cite{Iwaniec-Martin_book} and~\cite{Hencl-Koskela-book}.

A core part of the theory of quasiregular mappings is that the distortion estimate implies several strong topological properties. In higher dimensions this was pioneered
in a series of papers by Reshetnyak (1966--1969). Accordingly,
spatial quasiregular mappings enjoy the following properties: 
\begin{enumerate}[label=(\roman*)]
	\item\label{it:holder} A $K$-quasiregular mapping is locally $1/K$-H\"older continuous~\cite{Reshetnyak_continuity};
	\item\label{it:disopen} A nonconstant quasiregular mapping is both discrete and open~\cite{Reshetnyak_QROrigin, Reshetnyak_Theorem2};
	\item\label{it:liouville} A bounded quasiregular mapping in $\R^n$ is constant~\cite{Reshetnyak_Liouville}.
\end{enumerate}
Recall that a mapping $f$ is open if $f(U)$ is open for every open $U$, and that $f$ is discrete if the sets $f^{-1}\{y\}$ consist of isolated points. Generalizations of these results also hold for mappings of finite distortion, assuming sufficient integrability conditions on the distortion function; see e.g.\ \cite{Vodopanov-Goldstein-cont, Iwaniec-Koskela-Onninen-Invent, Iwaniec-Sverak, Manfredi-Villamor, Hencl-Rajala}.

Solutions to the heterogeneous distortion inequality~\eqref{eq:main_definition} generally lack  these powerful conditions, and therefore require new tools for their treatment. For instance, if $g \colon \Omega \to \R$ in $W^{1,n} (\Omega)$ and $f(x)= (e^{g(x)}, 0, \dots, 0)$, then $f$ satisfies \eqref{eq:main_definition} with $\sigma=|\nabla g|\in L^n(\Omega)$. This results in a discontinuous solution $f$ by simply choosing  a discontinuous $g \in W^{1,n} (\Omega)$. Similarly, by choosing a smooth, compactly supported $g$, we have that $f$ satisfies \eqref{eq:main_definition} with $\sigma\in L^p(\R^n)$ for every $p \in [1, \infty]$, yet $f$ is neither constant nor discrete or open. Hence, the heterogeneous distortion inequality even  with  a regular $\sigma$ is too weak for its solutions to satisfy the above conditions \ref{it:disopen} and~\ref{it:liouville}, and therefore the use of e.g.\ degree theory~\cite{Fonseca-Gangbo-book} is unavailable for non-quasiregular solutions.

However, in the planar case $n = 2$, Astala, Iwaniec and Martin showed that entire solutions $f$ of \eqref{eq:main_definition} which vanish at infinity  do satisfy a counterpart of the Liouville  theorem \ref{it:liouville}, provided that $\sigma$ is sufficiently regular; see~\cite[Theorem 8.5.1]{Astala-Iwaniec-Martin_Book}. This uniqueness theorem has found  several important applications, such as the nonlinear $\overline{\partial}$-problem in the theory of holomorphic motions~\cite{holomorphicmotion} or the solution to the Calder\'on problem~\cite{Astala-Paivarinta}. It is for this reason that Astala, Iwaniec and Martin asked  in \cite[p.\ 253]{Astala-Iwaniec-Martin_Book} whether their form of the Liouville theorem and a version of the continuity result~\ref{it:holder} remain valid in higher dimensions.

\vskip 0.2cm	
\begin{AIMproblem}  
	{\it Suppose that  $f\in W_{\loc}^{1,n} (\R^n , \R^n)$ satisfies the heterogeneous distortion inequality~\eqref{eq:main_definition} with $\sigma \in L^{n+\eps}(\R^n) \cap L^{n-\eps}(\R^n)$ for some $\eps>0$. Under these assumptions, is the mapping $f$ continuous? Moreover, does $f$ satisfy the following Liouville-type theorem: if $f(x) \to 0$ as $\abs{x} \to \infty$, then $f\equiv 0$?} 
\end{AIMproblem}
\vskip 0.2cm
	
Along with their proof of the planar case, Astala, Iwaniec and Martin  provided a counterexample which shows that the assumption $\sigma \in L^{2}(\C)$ is not enough for the Liouville-type theorem; see~\cite[Theorem 8.5.2.]{Astala-Iwaniec-Martin_Book}. Note
that in the planar case,  the heterogeneous distortion inequality \eqref{eq:main_definition} with $\sigma \in L^{n+\eps}(\R^n) \cap L^{n-\eps}(\R^n)$ amounts to saying that the solutions $f\in W_{\loc}^{1,2} (\Omega, \C)$ satisfy the homogeneous
differential inequality 
\begin{equation}\label{eq:AstalaIwaniecMartin_definition}
	\abs{\partial_{\overline{z}} f} \leq k \abs{\partial_{z} f} + \tilde{\sigma} \abs{f},
\end{equation}
with $k < 1$ and $\tilde{\sigma}  \in L^{2+\eps}(\C) \cap L^{2-\eps}(\C)$. Moreover,~\eqref{eq:AstalaIwaniecMartin_definition} can in fact be expressed as a  linear heterogeneous Cauchy-Riemann system, and is thus uniformly elliptic. However, for $n\ge 3$, the theory is  nonlinear, and
the main inherent difficulty lies in the lack of general existence theorems
and counterparts to power series expansions of the solutions.

In this paper, we give an affirmative answer to  the Astala--Iwaniec--Martin question. The continuity is the more straightforward part of the solution, whereas the main challenge lies in proving the Liouville-type uniqueness theorem.
	
\subsection{Continuity}  The solutions to~\eqref{eq:main_definition} are indeed continuous when $\sigma \in L_{\loc}^{n+\eps}(\Omega)$. We in fact establish  sharp local H\"older continuity estimates for such mappings. 

For $\gamma \in (0,1]$ we denote the class of locally  $\gamma$-H\"older continuous mappings $f \colon \Omega \to \R^n$ by  $C^{0, \gamma}_\loc(\Omega, \R^n)$. We let  $\gamma_K := 1/K$ denote the sharp H\"older exponent of $K$-quasiregular mappings, and  $\gamma_\eps := \eps/(n+\eps)$ the sharp H\"older exponent of $W^{1, n+\eps}$-functions.  Our result shows that if $\gamma_K \not = \gamma_\eps$, then the sharp H\"older exponent $\gamma$ for solutions to~\eqref{eq:main_definition} is the minimum of the exponents $\gamma_K$ and  $\gamma_\eps$. There is, however, a somewhat surprising special case: if $\gamma_K$ and $\gamma_\eps$ coincide, then we in fact end up with an infinitesimally weaker H\"older exponent than $\gamma_K = \gamma_\eps$.
	
\begin{thm}\label{thm:Holder}
	Let $\Omega \subset \R^n$ be a domain,  $K \in [1, \infty)$, $\eps > 0$,  $\gamma= \min (\gamma_K, \gamma_\eps)$, and  $\sigma \in L_{\loc}^{n+\eps}(\Omega)$. Suppose that  $f \in W^{1,n}_\loc(\Omega, \R^n)$ satisfies the heterogeneous distortion inequality \eqref{eq:main_definition} with $K$ and $\sigma$.
		
	If $\gamma_K \neq \gamma_\eps$, then $f \in C^{0, \gamma}_\loc(\Omega, \R^n)$. Moreover, there exist $\sigma$ and $f$ satisfying the assumptions such that $f \notin C^{0, \gamma'}_\loc(\Omega, \R^n)$ for any $\gamma' > \gamma$.
		
	If $\gamma_K = \gamma_\eps$, then $f \in C^{0, \gamma'}_\loc(\Omega, \R^n)$ for every $\gamma' < \gamma$, and there exist $\sigma$ and $f$ satisfying the assumptions such that $f \notin C^{0, \gamma}_\loc(\Omega, \R^n)$.
\end{thm} 

\subsection{Liouville Theorem}
Recall that the classical Liouville theorem asserts that bounded entire holomorphic functions are constant~\cite{Cauchy}. Its quasiregular counterpart~\ref{it:liouville} follows from the Caccioppoli inequality~\cite{Caccioppoli}, which controls the derivatives of a quasiregular mapping locally in terms of the mapping. Attempting the same approach under the assumptions of the Astala-Iwaniec-Martin question yields that the integral of the Jacobian over the entire space $\R^n$ equals zero, or equivalently, that
\begin{equation}\label{eq:natural}
	\int_{\R^n} \abs{Df}^n \le \int_{\R^n} \abs{f}^n \sigma ^n. 
\end{equation}
Analogously to the quasiregular theory~\cite{Gehring}, the local variants of the \emph{energy estimate}~\eqref{eq:natural} imply a higher degree of integrability for the derivatives of a solution. These observations can be further combined with the nonlinear Hodge theory developed by Iwaniec and Martin~\cite{Iwaniec-Martin_Acta, Iwaniec_Annals}, which provides $L^p$-estimates for $\abs{Df}$ with exponents $p$ either smaller or larger than the dimension $n$. However, this approach alone appears insufficient for the desired Liouville-type theorem, and although it could be used to show the continuity of $f$, the proof would not yield the sharp H\"older exponents we obtain in Theorem \ref{thm:Holder}.

Our first Liouville-type result shows that if $\sigma \in L^n(\R^n) \cap L^{n+\eps}_\loc(\R^n)$, then nontrivial solutions to \eqref{eq:main_definition} that are bounded in $\R^n$ do not
vanish at any point. Note that the standard radial example of a $K$-quasiregular mapping $f(x)=\abs{x}^{\frac{1}{K}-1}x$ shows that  this result has no local alternative.
	
\begin{thm}\label{thm:nonzero}
	Suppose that $f \in W^{1,n}_\loc(\R^n, \R^n)$ satisfies the heterogeneous distortion inequality \eqref{eq:main_definition} with $K \in [1, \infty)$ and $\sigma \in L^{n}(\R^n) \cap L^{n+\eps}_\loc(\R^n)$ for some $\eps > 0$. If $f$ is bounded and $f$ is not identically zero, then $0 \notin f(\R^n)$.
\end{thm}
	
The first and key step in the proof of Theorem~\ref{thm:nonzero} is to show that a solution to  the heterogeneous distortion inequality with $\sigma \in L^n (\R^n)$ satisfies
\begin{equation}\label{eq:naturalestimatelog}
	\int_{\R^n} \frac{J_f}{\abs{f}^n}  =0, \qquad 
	\textnormal{and therefore,} \quad
	\int_{\R^n} \frac{\abs{Df}^n}{\abs{f}^n} \le \int_{\R^n}  \sigma^n \, .
\end{equation}
The proof of this is based on a thorough analysis of the integrals of $J_f$ over the sublevel sets of $\abs{f}$. 
In turn, we obtain that $\abs{\nabla \log \abs{f}} \in L^n (\R^n)$. The second step is to establish a Morrey-type decay estimate~~\cite{Morrey-Holder} on the integrals of $\abs{\nabla \log \abs{f}}^n$ over the balls in $\R^n$. Our proof of the decay estimate relies on the formal identity
\[  \frac{J_f(x)\,}{\abs{f(x)}^n} \vol_n  = \dd \omega,\]
where $\omega$ is a certain differential $(n-1)$-form.  The resulting polynomial decay estimate implies that the function $\log \abs{f}$ is   H\"older continuous. Thus, the mapping  $f$ itself must omit the point $0$ from its range, completing the proof of Theorem \ref{thm:nonzero}. 

Our final main result is then the Liouville part of the Astala-Iwaniec-Martin question.

\begin{thm}\label{thm:Liouville}
	Suppose that $f \in W^{1,n}_\loc(\R^n, \R^n)$ satisfies the heterogeneous distortion inequality \eqref{eq:main_definition} for $K \in [1, \infty)$ and $\sigma \in L^{n-\eps}(\R^n) \cap L^{n+\eps}(\R^n)$ for some $\eps > 0$. If  $\lim_{x \to \infty} \abs{f(x)} = 0$, then $f \equiv 0$.
\end{thm}

 We prove Theorem~\ref{thm:Liouville} by showing that, if there exists a non-identically zero solution $f \colon \R^n \to \R^n$ to~\eqref{eq:main_definition} with $\sigma \in L^{n+\eps}(\R^n) \cap L^{n-\eps}(\R^n)$, then the oscillation of the function $\log \abs{f}$ over the entire space $\R^n$ is uniformly bounded. This clearly leads to a contradiction  with the assumption $f(x) \to 0$ when $\abs{x} \to \infty$. The crucial step in obtaining the uniform oscillation bound is to strengthen the estimate~\eqref{eq:naturalestimatelog}; that is,  to prove an integrability estimate below the natural exponent $n$ for the expression $\abs{Df}/\abs{f}$. 
 
Our solution is influenced by the case $n = 2$. Indeed, in this case, the mapping $f$ has no zeros by Theorem \ref{thm:nonzero}, and hence $f$ has a well-defined complex logarithm $\log f$. The mapping $\log f$ satisfies the distortion estimate
\begin{equation}\label{eq:logestimate}
	\abs{D \log f}^2 \leq K J_{\log f} + \sigma^2
\end{equation}
almost everywhere. This, in turn, gives a nonhomogeneous linear
elliptic equation for $\log f$, which implies the desired integrability
estimate below the natural exponent for $\abs{D \log f}$. In higher dimensions,
the issue is  to construct a similar map $\log f \colon \R^n \to \R^n$.  The Zorich map $h_Z \colon \R^n \to \R^n \setminus \{0\}$ provides a well-known $n$-dimensional generalization of the planar exponential mapping; see~\cite{Zorich}. Unfortunately, $h_Z$ has a branch set consisting of  $\left(n-2\right)$-dimensional hyperplanes, which prevents lifting an arbitrary continuous $f \colon \R^n \to \R^n \setminus \{0\}$ through $h_Z$.

We circumvent the lifting difficulties by moving to the Riemannian manifold setting. Indeed, a well defined counterpart for the logarithm exists from $\R^n \setminus \{0\}$ to $\R \times \S^{n-1}$, providing us with a mapping $``\log f\text{''} \colon \R^n \to \R \times \S^{n-1}$. This mapping satisfies a higher dimensional counterpart of the estimate~\eqref{eq:logestimate}, and the single Euclidean component $\R$ in the target space is sufficient for a Caccioppoli-type inequality to hold, which leads to the desired integrability estimate below the natural exponent $n$.
	
\section*{Acknowledgments}
We thank Tadeusz Iwaniec and Xiao Zhong for discussions and shared insights.

\section{Preliminaries}

In this section, we go over some of the tools we require which might be less familiar to readers.

\subsection{Sobolev mappings with manifold target}

For the most part of this text we use the standard Sobolev spaces $W^{1,p}(\Omega, \R^k)$ and $W^{1,p}_\loc(\Omega, \R^k)$, where $\Omega \subset \R^n$ is an open connected set, see e.g.~\cite{Fonseca-Gangbo-book, Ziemer-book}. However, towards the end of the text, we consider a locally Sobolev mapping $f \colon \R^n \to M$, where $M$ is an $n$-dimensional Riemannian manifold.

There are various approaches to defining first order Sobolev mappings with a manifold target; see e.g.\ \cite{Hajlasz-Iwaniec-Maly-Onninen} or \cite{Convent-VanSchaftingen_Sobolev}. However, in our case, we only have to consider continuous Sobolev mappings with a manifold target, which simplifies the definition significantly.

\begin{defn}
	Let $\Omega \subset \R^n$ be a domain, and let $M$ be a Riemannian $k$-manifold. We say that a continuous $f \colon \Omega \to M$ is in the Sobolev space $W^{1,p}_\loc(\Omega, M)$ for $p \in [1, \infty)$ if, for every $x \in \Omega$, there exists a neighborhood $U \subset \Omega$ of $x$ and a smooth bilipschitz chart $\varphi \colon V \to M$ such that $f U \subset \varphi V$ and $\varphi^{-1} \circ f \in W^{1,n}(U, \R^k)$.
\end{defn}

If a continuous function $f \colon \Omega \to M$ is in $W^{1,p}_\loc(\Omega, M)$, then there exists a weak derivative $Df \colon \Omega \times \R^n \to TM$ which satisfies $D(\varphi^{-1} \circ f) = D(\varphi^{-1}) \circ Df$ for bilipschitz charts $\varphi \colon V \to M$. This weak derivative is unique up to a set of measure zero, in the sense that if $\tilde{D}f$ is another such mapping, then $\tilde{D}f = Df$ outside a set of the form $E \times \R^n$ where the set $E$ has zero $n$-dimensional Lebegue measure, $m_n(E) = 0$. 

At a given point $x \in \Omega$, we denote by $\abs{Df(x)}$ the operator norm of $Df(x) \colon \R^n \to T_{f(x)} M$, where $T_{f(x)} M$ is equipped with the norm induced by the Riemannian metric. It follows that $\abs{Df} \colon \Omega \to [0, \infty]$ is in $L^p_\loc(\Omega)$ for any continuous $f \in W^{1,p}_\loc(\Omega, M)$. If $\dim \Omega = \dim M$ and $M$ is oriented, then we also have a measurable Jacobian $J_f \colon \Omega \to \R$, characterized almost everywhere by $f^* \vol_M = J_f \vol_n$.

We remark that if $M = \R^k$, then the above definition coincides with the usual definition of $f \in W^{1,p}_\loc(\Omega, \R^k)$ for continuous $f$. We also remark that if the target is a product manifold $M = M_1 \times M_2$, then given two continuous mappings $f_1 \colon \Omega \to M_1$ and $f_2 \colon \Omega \to M_2$, we have that $(f_1, f_2) \in W^{1,p}_\loc(\Omega, M)$ if and only if $f_1 \in W^{1,p}_\loc(\Omega, M_1)$ and $f_2 \in W^{1,p}_\loc(\Omega, M_2)$.

Next we  recall Sobolev differential forms. Namely,  suppose that $M$ is a Riemannian manifold. A measurable $k+1$-form $d\omega \in L^1_\loc(\wedge^{k+1} M)$ is the \emph{weak differential} of a measurable $k$-form $\omega \in L^1_\loc(\wedge^{k} M)$ if
\[
	\int_M \omega \wedge d\eta = (-1)^{k+1} \int_M d\omega \wedge \eta
\]
for every $\eta \in C^\infty_0(\wedge^{n-k-1} M)$. We denote by $W^{d, p, q}_\loc(\wedge^k M)$ the space of $k$-forms $\omega \in L^p_\loc(\wedge^k M)$ with a weak differential $d\omega \in L^q_\loc(\wedge^{k+1} M)$. The version where $\omega \in L^p(\wedge^k M)$ and $d\omega \in L^q(\wedge^{k+1} M)$ is denoted $W^{d, p, q}(\wedge^k M)$. We also use the shorthands $W^{d, p}(\wedge^k M) = W^{d, p, p}(\wedge^k M)$ and $W^{d, p}_\loc(\wedge^k M) = W^{d, p, p}_\loc(\wedge^k M)$.

In particular, we require the following standard result about pull-backs of compactly supported smooth forms with Sobolev mappings. We sketch the proof for the convenience of the reader.

\begin{lemma}\label{lem:diff_form_Sobolev_pullback}
	Let $\Omega \subset \R^n$ be a domain, and let $M$ be a Riemannian $m$-manifold. Suppose that $f \in W^{1,p}_\loc(\Omega, M)$, where we assume that $f$ is continuous if $M \neq \R^{m}$. If $\omega \in C^\infty_0(\wedge^k M)$ and $p \geq k + 1$, then $f^* \omega \in W^{d, p/k, p/(k+1)}_\loc(\wedge^k \Omega)$ and $d f^* \omega = f^* d \omega$.
\end{lemma}
\begin{proof}[Sketch of proof]
	The fact that $f^*\omega$ and $f^*d\omega$ satisfy the correct integrabilities follows from the estimates
	\begin{align*}
		\abs{(f^* \omega)_x} &\leq \norm{\omega}_\infty \abs{Df(x)}^k,&
		\abs{(f^* d\omega)_x} &\leq \norm{d\omega}_\infty \abs{Df(x)}^{k+1}
	\end{align*}
	for a.e.\ $x \in \Omega$. 
	
	For $d f^* \omega = f^* d \omega$, it suffices to consider $\omega$ for which the support $\spt \omega$ of $\omega$ is contained in the domain of a bilipschitz chart $\phi \colon U \to \R^m$, as the general $\omega$ is a finite sum of such forms. Moreover, it suffices to consider $\omega$ of the form $\omega_0 d\phi_{1} \wedge \dots \wedge d\phi_k$, as a general $\omega$ with $\spt \omega \subset U$ is again a finite sum of such forms with the coordinates of $\phi$ rearranged. We may also select $\phi' \in C^\infty_0(M, \R^m)$ such that $\phi' = \phi$ on $\spt \omega$, which lets us write $\omega = \omega_0 d\phi_{1}' \wedge \dots \wedge d\phi_k'$ in a form where the components are defined on all of $N$.
	
	We then use the chain rule of locally Sobolev and $C^1$ maps to conclude that$f^* \omega_0 = \omega_0 \circ f \in L^p_\loc(\wedge^0 \Omega)$, $f^* d\omega_0 = d (\omega_0 \circ f) \in L^p_\loc(\wedge^1 \Omega)$, and $f^* d\phi_i' = d (\phi_i' \circ f) \in L^p_\loc(\wedge^1 \Omega)$ for every $i \in \{1, \dots, k\}$. By the wedge product rules for Sobolev forms and the formula $d \circ d = 0$ for the weak differential, we conclude that $f^* d\omega = f^* d\omega_0 \wedge f^*d\phi_1' \wedge \dots \wedge f^* d \phi_k' = d(f^*\omega_0) \wedge d(\phi_1' \circ f) \wedge \dots \wedge d (\phi_k' \circ f) = d f^* \omega$.
\end{proof}

\subsection{Caccioppoli inequality}

The Caccioppoli inequalities are  a standard tool in the study of quasiregular mappings.  The most basic form of the Caccioppoli estimate for a  $K$-quasiregular mapping  $f \colon \Omega \to \R^n$ reads as
\[\int_\Omega \eta^n \abs{Df}^n \le n^n K^n  \int_\Omega \abs{f}^n\ \abs{\nabla \eta}^n \, , \]
where is a real-valued smooth test function with compact support in $\Omega$. This follows from the general inequality
\[\int_\Omega \eta^n J_f \le n \int_\Omega \abs{Df}^{n-1} \abs{\eta}^{n-1} \abs{\nabla \eta} \abs{f}\]
which can be proved for arbitrary $f\in W_{\loc}^{1,n}(\Omega, \R^n)$ via an integration-by-parts argument. Our arguments, however, require a version with a target space other than $\R^n$. 
In general, it is not possible to obtain a Caccioppoli-type estimate for mappings $f \colon \R^n \to M$ where $M$ is a Riemannian $n$-manifold.  This happens when $M$ is a \emph{rational homology sphere}, see~\cite{Hajlasz-Iwaniec-Maly-Onninen}. However, the standard proof generalizes to  the case of $M = \R \times N$, where $N$ is a Riemannian $(n-1)$-manifold.

\begin{lemma}\label{lem:caccioppoli}
	Let $\Omega \subset \R^n$ be a domain, and let $N$ be a compact oriented Riemannian $(n-1)$-manifold without boundary. Let $f \in W^{1,n}_\loc(\Omega, \R \times N)$, where we assume that $f$ is continuous if $N \neq \R^{n-1}$. Denote by $f_\R \colon \Omega \to \R$ and $f_N \colon \Omega \to N$ the coordinate functions of $f$. Then for every $\eta \in C^\infty_0(\Omega)$ and every $c \in \R$, we have
	\[
		\abs{\int_{\Omega} \eta^n J_f} \leq n \int_{\R^n} \abs{Df}^{n-1} \abs{\eta}^{n-1} \abs{\nabla \eta} \abs{f_{\R} - c}.
	\]
\end{lemma}
\begin{proof}
	We define the function $f_\eta = (\eta^n (f_\R-c), f_N)$. Then $f_\eta \in W^{1,n}_\loc(\Omega, \R \times N)$. Moreover, since $\vol_{\R \times N} = (\pi_\R^* \vol_1)\wedge(\pi_N^* \vol_N)$, we have
	\[
		f_\eta^* \vol_{\R \times N}
		= d(\eta^n (f_\R - c)) \wedge f_N^* \vol_N.
	\]
	
	We may select a subdomain $U$ with smooth boundary such that $U$ is compactly contained in $\Omega$ and $\spt \eta \subset U$. Since $\eta^n (f_\R - c) \in W^{1,n}(U)$ with compact support, we may approximate it in $W^{1,n}(U)$ with $g_i \in C^\infty_0(U)$. Moreover, we have by Lemma \ref{lem:diff_form_Sobolev_pullback} that $f_N^* \vol_N \in W^{d, n/(n-1), 1}_\loc(\wedge^{n-1} \Omega)$ and $d f_N^* \vol_N = f_N^* d\vol_N = 0$. Hence, $f_N^* \vol_N \in W^{d, n/(n-1)}(\wedge^{n-1} U)$, and we may approximate $f_N^* \vol_N$ in $W^{d, n/(n-1)}(\wedge^{n-1} U)$ with $\omega_i \in C^\infty(\wedge^{n-1} U)$ (see e.g. \cite[Corollary 3.6]{Iwaniec-Scott-Stroffolini}).
	
	By a standard H\"older-type estimate and the Leibniz rule, it therefore follows that $d(g_i \omega_i) \to d(\eta^n (f_\R - c)) \wedge f_N^* \vol_N$ in $L^1(\wedge^n U)$. However, since $g_i \omega_i$ is smooth and compactly supported for every $i$, it follows that
	\[
		\int_{\Omega} d(\eta^n (f_\R - c)) \wedge f_N^* \vol_N
		= \lim_{i \to \infty} \int_{U} d(g_i \omega_i)
		= \lim_{i \to \infty} 0
		= 0.
	\]
	Hence, we may estimate that
	\begin{multline*}
		\abs{\int_{\Omega} \eta^n J_f}
		= \abs{\int_{\Omega} \eta^n d(f_\R - c) \wedge f_N^* \vol_N}
		= \abs{\int_{\Omega} (f_\R - c) d(\eta^n) \wedge f_N^* \vol_N}\\
		\leq \int_{\R^n} \abs{f_{\R} - c} (n\abs{\eta}^{n-1} \abs{\nabla \eta}) \abs{Df}^{n-1}.
	\end{multline*}
	The claim therefore follows.
\end{proof}

\subsection{Jacobians of entire mappings}

To end the preliminaries section, we also discuss a result regarding the Jacobian of an entire Sobolev mapping. It is the main reason why we stated the Caccioppoli inequality also in the case where the target is the product of $\R$ and an $(n-1)$-manifold.

\begin{lemma}\label{lem:zero_Jacobian}
	Let $N$ be a compact oriented Riemannian $(n-1)$-manifold without boundary. Suppose that $f \in W^{1,n}_\loc(\R^n, \R \times N)$, where we assume that $f$ is continuous if $N \neq \R^{n-1}$. If $\abs{Df} \in L^n(\R^n)$, then
	\[
		\int_{\R^n} J_f = 0.
	\]
\end{lemma}
\begin{proof}
	We let $\eta_r \in C^\infty(\R^n,[0,1])$ be such that we have $\eta\vert_{B^n(0, r)} \equiv 1$, $\eta \vert_{\R^n \setminus B^n(0, 2r)} \equiv 0$, and $\abs{\nabla \eta} \leq 2/r$. We again denote $f = (f_{\R}, f_{N})$. We then use the Caccioppoli inequality of Lemma \ref{lem:caccioppoli} and H\"older's inequality to obtain
	\begin{align*}
		&\abs{\int_{\R^n} \eta_r^n J_f}\\
		&\qquad\leq \int_{\R^n} \eta_r^{n-1}\abs{Df}^{n-1}
			\abs{f_{\R} - (f_{\R})_{B^n(0, 2r)}} \abs{\nabla \eta_r}\\
		&\qquad\leq \left(\int_{\spt \abs{\nabla \eta_r}} 
			\eta_r^n\abs{Df}^n \right)^\frac{n-1}{n}
			\left( \int_{B^n(0, 2r)} \abs{f_{\R} - (f_{\R})_{B^n(0, 2r)}}^n 
			\abs{\nabla \eta_r}^n \right)^\frac{1}{n}
	\end{align*}
	Since $f_\R \in W^{1,n}_\loc(\R^n)$ and $\abs{\nabla \eta} \leq 2/r$, the Sobolev-Poincar\'e inequality then yields that
	\begin{align*}
		&\left(\int_{\spt \abs{\nabla \eta_r}} 
			\eta_r^n\abs{Df}^n \right)^\frac{n-1}{n}
			\left( \int_{B^n(0, 2r)} \abs{f_{\R} - (f_{\R})_{B^n(0, 2r)}}^n 
			\abs{\nabla \eta_r}^n \right)^\frac{1}{n}\\
		&\qquad\leq 4 \left(\int_{\spt \abs{\nabla \eta_r}} 
			\eta_r^n\abs{Df}^n \right)^\frac{n-1}{n}
			\left( \frac{1}{(2r)^n}\int_{B^n(0, 2r)} 
			\abs{f_{\R} - (f_{\R})_{B^n(0, 2r)}}^n \right)^\frac{1}{n}\\
		&\qquad\leq 4C_n\left(\int_{\R^n \setminus B^n(0, r)}
			\abs{Df}^n \right)^\frac{n-1}{n}
			\left( \int_{B^n(0, 2r)} \abs{\nabla f_{\R}}^n \right)^\frac{1}{n}.
	\end{align*}
	Since $\abs{\nabla f_{\R}} \leq \abs{Df} \in L^n(\R^n)$, the first integral term on the right hand side tends to 0 as $r \to \infty$, while the second term stays bounded. Since $\abs{\eta_r^n J_f} \leq \abs{Df}^n$, the claim therefore follows by dominated convergence.
\end{proof}

\section{H\"older continuity}\label{sect:Holder}
In this section, we prove the continuity part of Theorem \ref{thm:Holder}.  Our proof is based on  Morrey's rather elegant ideas in geometric function theory~\cite{Morrey-Holder, Reshetnyak_continuity, Iwaniec-Martin_book}. A crucial tool in establishing the sharp H\"older exponent is the isoperimetric inequality in the Sobolev space $W_{\loc}^{1,n} (\Omega, \R^n)$. For $x\in \Omega$ and almost every $r>0$ such that $B_r=B^n(x,r)$ compactly contained in $\Omega$, we have
\begin{equation}
\label{eq:isoperimetric}
\left|  \int_{B_r} J_f \right| \le (n \sqrt[n-1]{\omega_{n-1}})^{-1} 
\left(\int_{\partial B_r} \abs{D^\sharp f} \right)^\frac{n}{n-1},
\end{equation}
where $\omega_{n-1}$ is the $(n-1)$-dimensional area of the unit sphere $\partial B_1$ in $\R^n$. Here $D^\sharp f (x)$ stands for the cofactor matrix of the differential matrix  $Df(x)$. For a diffeomorphism $f \colon B_r \to U$ this integral form  of the isoperimetric inequality follows immediately  from  the familiar geometric form of the isoperimetric inequality
\[
n^{n-1} \omega_{n-1} [m_n(U)]^{n-1} \le [m_{n-1}(\partial U)]^n,
\]
where $m_n(U)$ stands for the volume of a domain $U\subset \R^n$ and $m_{n-1}({\partial U})$ is its $(n-1)$-dimensional surface area. For the proof of~\eqref{eq:isoperimetric} for Sobolev mappings see  Reshetnyak \cite[Lemma II.1.2.]{Reshetnyak-book} for a more detailed account.

We begin with the primary estimate our proof relies on.
\begin{lemma}\label{lem:local_integral_ineq}
	Let $\Omega, \Omega' \subset \R^n$ be bounded domains with $\overline{\Omega} \subset \Omega'$. Suppose that $f \in W^{1,n}(\Omega', \R^n)$ satisfies the heterogeneous distortion inequality \eqref{eq:main_definition} for $K \in [1, \infty)$ and $\sigma \in L^n(\Omega') \cap L^{n+\eps}(\Omega')$, where $\eps > 0$. Let $x \in \Omega$, let $B_r = B(x, r)$ for all $r \in (0, \infty)$, let $R > 0$ be such that $B_R \subset \Omega$. Then for every $\delta < \eps$ and a.e.\ $r < R$ we have the estimate
	\[
		\int_{B_r} \abs{Df}^n 
		\leq \frac{Kr}{n} \int_{\partial B_r} \abs{Df}^n
		+ C r^\frac{n\delta}{n+\delta},
	\]
	where $C$ depends only on $n$, $\Omega$, $f$, $\sigma$, $\eps$ and $\delta$. In particular, $C$ doesn't depend on $x$, $r$ and $R$. Moreover, if $f \in L^\infty(\Omega, \R^n)$, then the estimate also holds for $\delta = \eps$.
\end{lemma}
\begin{proof}
	By using the heterogeneous distortion inequality, the isoperimetric inequality~\eqref{eq:isoperimetric} for $W^{1,n}$-mappings, Hadamard's inequality $\abs{D^\sharp f} \le \abs{Df}^{n-1}$, and H\"older's inequality, we obtain for a.e.\ $r < R$ the estimate
	\begin{align*}
		\int_{B_r} \abs{Df}^n
		&\leq K \int_{B_r} J_f + \int_{B_r} \abs{f}^n \sigma^n\\
		&\leq \frac{K}{n \sqrt[n-1]{\omega_{n-1}}}
			\left(\int_{\partial B_r} \abs{Df}^{n-1}\right)^\frac{n}{n-1}
			+ \int_{B_r} \abs{f}^n \sigma^n\\
		&\leq \frac{Kr}{n}
			\int_{\partial B_r} \abs{Df}^{n}
		+ \int_{B_r} \abs{f}^n \sigma^n
	\end{align*}
	For the final term, we note that for every $p < \infty$, we have $f \in L^p_\loc(\Omega', \R^n)$ by the Sobolev embedding theorem, and consequently also $f \in L^p(\Omega, \R^n)$. Hence, we may use H\"older's inequality to obtain the desired estimate
	\begin{align*}
		\int_{B_r} \abs{f}^n \sigma^n
		&\leq \left( \int_{B_r} \sigma^{n+\eps} \right)^\frac{n}{n+\eps}
		\left( \int_{B_r} 1 \right)^\frac{\delta}{n-\delta}
		\left( \int_{B_r} 
			\abs{f}^\frac{(n+\delta)(n+\eps)}{\eps - \delta}
			\right)^\frac{n(\eps - \delta)}{(n+\delta)(n+\eps)}\\
		&\leq \left( \int_{\Omega} \sigma^{n+\eps} \right)^\frac{n}{n+\eps}
		\left( \int_{\Omega} 
			\abs{f}^\frac{(n+\delta)(n+\eps)}{\eps - \delta}
			\right)^\frac{n(\eps - \delta)}{(n+\delta)(n+\eps)}
		r^\frac{n\delta}{n + \delta}.
	\end{align*}
	Moreover, if we additionally know that $f \in L^\infty(\Omega, \R^n)$, we obtain the claim for $\delta = \eps$ by estimating
	\[
		\int_{B_r} \abs{f}^n \sigma^n
		\leq \norm{f}_\infty 
			\left( \int_{\Omega} \sigma^{n+\eps} \right)^\frac{n}{n+\eps}
			r^\frac{n\eps}{n + \eps}.	
	\]
\end{proof}

Note that the estimate of Lemma \ref{lem:local_integral_ineq} is of the form $\Phi(r) \leq A r \Phi'(r) + B r^\alpha$. 
This differential inequality  allows us to obtain an estimate for the decay of $\Phi$ at 0, which in our case is a decay estimate on the integrals of  $\abs{Df}^n$ over balls. 

\begin{lemma}\label{lem:diff_ineq_solution_r}
	Suppose that $\Phi \colon [0, R] \to [0, S]$ is an absolutely continuous increasing function such that $\Phi(0) = 0$ and
	\begin{equation}\label{eq:Phi_diff_ineq_easy}
		\Phi(r) \leq A r \Phi'(r) + B r^\alpha
	\end{equation}
	for a.e.\ $r \in [0, R]$, where $A, \alpha > 0$ and $B \geq 0$. Then there exists a constant $C = C(A, B, \alpha, R, S)$ such that the following holds:
	\begin{itemize}
		\item if $\alpha < A^{-1}$, then for all $r \in [0, R]$ we have
		\[
			\Phi(r) \leq C r^{\alpha};
		\]
		\item if $\alpha = A^{-1}$, then for all $r \in [0, R]$ we have
		\[
			\Phi(r) \leq C r^{\alpha} \log\Bigl(\frac{Re^{1+\alpha^{-1}}}{r}\Bigr);
		\]
		\item if $\alpha > A^{-1}$, then for all $r \in [0, R]$ we have
		\[
			\Phi(r) \leq C r^{A^{-1}}.
		\]
	\end{itemize}
\end{lemma}
\begin{proof}
	We observe that
	\[
		\frac{\dd}{\dd r} \left(- A r^{-A^{-1}} \Phi(r)\right)
		= \frac{\Phi(r) - A r \Phi'(r)}{r^{1 + A^{-1}}}
	\]
	for a.e.\ $r \in [0, R]$. Consequently, the estimate \eqref{eq:Phi_diff_ineq_easy} can be rewritten in the form
	\[
		\frac{\dd}{\dd r} \left(- A r^{-A^{-1}} \Phi(r)\right)
		\leq B r^{-1 -(A^{-1} - \alpha)}.
	\]
	We integrate this estimate, obtaining
	\begin{equation}\label{eq:integrated_estimate}
		\int_r^R \frac{\dd}{\dd s} 
			\left(- A s^{-A^{-1}} \Phi(s)\right) \dd s
		\leq B \int_r^R s^{-1 -(A^{-1} - \alpha)} \dd s.
	\end{equation}
	Consider first the case $\alpha < A^{-1}$. Computing the integrals in \eqref{eq:integrated_estimate} yields
	\[
		A(r^{-A^{-1}} \Phi(r) - R^{-A^{-1}} \Phi(R))
		\leq \frac{B}{A^{-1} - \alpha} 
			\left( r^{-(A^{-1} - \alpha)}-R^{-(A^{-1} - \alpha)}\right),
	\]
	and further rearrangement and estimation yields
	\[
		\Phi(r) 
		\leq r^{A^{-1}} \left( R^{-A^{-1}} \Phi(R) 
			+ \frac{B}{1 - A\alpha} r^{\alpha - A^{-1}} \right)
		\leq \left( R^{-\alpha} S 
			+ \frac{B}{1 - A\alpha} \right) r^\alpha
	\]
	Suppose then that $\alpha = A^{-1}$. Then \eqref{eq:integrated_estimate} results in
	\[
		A(r^{-A^{-1}} \Phi(r) - R^{-A^{-1}} \Phi(R))
		\leq B \log(R/r),
	\]
	and we may again further estimate
	\begin{multline*}
		\Phi(r) 
		\leq r^{A^{-1}} \left( R^{-A^{-1}} \Phi(R) 
			+ \frac{B}{A} \log \frac{R}{r}\right)\\
		\leq \left( R^{-A^{-1}} S + \frac{B}{A} \right) 
			r^{A^{-1}} \log\Bigl(\frac{Re^{1+\alpha^{-1}}}{r}\Bigr).
	\end{multline*}
	Finally, consider the case $\alpha > A^{-1}$. In this case, it follows from \eqref{eq:integrated_estimate} that
	\[
		A(r^{-A^{-1}} \Phi(r) - R^{-A^{-1}} \Phi(R))
		\leq \frac{B}{\alpha - A^{-1}} 
		\left( R^{\alpha - A^{-1}} - r^{\alpha - A^{-1}}\right),
	\]
	and further rearrangement and estimation yields
	\[
		\Phi(r) \leq r^{A^{-1}} \left(
			R^{-A^{-1}} S + \frac{B R^{\alpha - A^{-1}}}{A\alpha - 1}
			\right).
	\]
\end{proof}

For the remaining component to the proof of Theorem \ref{thm:Holder}, we recall a well known fact that the decay estimate on $\abs{Df}$ implies that $f$ belongs to a  Morrey--Campanato space~\cite{Morrey-space, Campanato-space}, and is thus H\"older continuous. The precise formulation of this fact that we use is as follows. 

\begin{lemma}\label{lem:Holder_continuity_general_log}
	Let $\Omega = B^n(x, R/4)$ for some $R > 0$ and $k \in \N$. Suppose that $f \in W^{1,n}(B^n(x, R), \R^k)$ satisfies 
	\begin{equation}\label{eq:decay_assumption}
		\int_{B_r} \abs{Df}^n \leq C r^\alpha \log^\beta \frac{L}{r}
	\end{equation}
	for all $B_r = B^n(y, r) \subset B^n(x, R)$, where $\alpha > 0$, $\beta \geq 0$, and $L > 0$ is large enough that $R < L e^{-\beta/\alpha}$. Then
	\[
		\abs{f(y) - f(z)} \leq C' \abs{y-z}^{\frac{\alpha}{n}} \log^\frac{\beta}{n} \frac{L}{4\abs{y - z}}
	\]
	for all $y, z \in \Omega$, where $C'$ depends on $n$, $k$, $C$, $A$, $\alpha$ and $\beta$.
\end{lemma}

Note that the assumption $R < L e^{-\beta/\alpha}$ is to ensure that $r^\alpha \log^\beta (A/r)$ is increasing on $[0, R]$. Lemma~\ref{lem:Holder_continuity_general_log} is merely a small variant of a classical result of Morrey \cite{Morrey-Holder} with an extra logarithmic term, where the logarithmic term becomes relevant when investigating the exact modulus of continuity. See \cite[Theorem 3.5.2]{Morrey-Book} for a proof in the classical case $\beta = 0$. For general $\beta$, we note that Lemma~\ref{lem:Holder_continuity_general_log} also follows from the fractional maximal function estimate of Sobolev functions: if $u \in W^{1,1}_\loc(\R^n)$ and $\gamma \in (0,1)$, then for all $y,z \in \R^n$ outside a set of measure zero, we have
\begin{multline}\label{eq:frac_maximal_estimate}
	\abs{u(y)-u(z)}\\
	\le C_{n, \gamma} \abs{y-z}^{1-\gamma} \left(\maxfun_{\gamma, 4\abs{y-z}} \abs{Du}(y)
		+ \maxfun_{\gamma, 4\abs{y-z}} \abs{Du}(z) \right),
\end{multline}
where $\maxfun_{\gamma, R}$ stands for the \emph{restricted fractional maximal function}
\[
	\maxfun_{\gamma, R} \abs{Du}(y) = \sup_{0<r<R} \frac{r^\gamma}{m_n(B(y,r))} \int_{B(y,r)} \abs{Du}.
\]
Indeed, taking $\gamma$ close to $1$ and combining \eqref{eq:decay_assumption} with \eqref{eq:frac_maximal_estimate} yields the desired estimate of Lemma \ref{lem:Holder_continuity_general_log}.
The proof of \eqref{eq:frac_maximal_estimate} is due to Hedberg~\cite{Hedberg}.

We are now ready to prove the local H\"older continuity stated in Theorem \ref{thm:Holder}.

\begin{lemma}\label{lem:Holder_continuity}
	Let $\Omega = B^n(x, R)$ for some $R > 0 $. Suppose that $f \in W^{1,n}(B^n(x, 5R), \R^n)$ satisfies the heterogeneous distortion inequality \eqref{eq:main_definition} with $K \in [1, \infty)$ and $\sigma \in L^n(B^n(x, 5R)) \cap L^{n+\eps}(B^n(x, 5R))$, where $\eps > 0$. 
	
	If $K^{-1} \neq \eps/(n+\eps)$, then
	\[
		\abs{f(y) - f(z)} \leq C \abs{y-z}^{\min(K^{-1}, \frac{\eps}{n+\eps})}
	\]
	for all $y, z \in \Omega$, where $C = C(n, K, \eps, \sigma, f)$. 
	
	If  $K^{-1} = \eps/(n+\eps)$, then
	\[
		\abs{f(y) - f(z)} \leq C \abs{y-z}^{K^{-1}} \log^{\frac{1}{n}} \Bigl(\frac{Re^{1+K}}{\abs{y-z}}\Bigr)
	\]
	for all $y, z \in \Omega$, where $C = C(n, K, \eps, \sigma, f)$.
\end{lemma}
\begin{proof}
	We first prove a slightly weaker H\"older continuity estimate for $f$ than is claimed. This in turn implies the local boundedness of $f$, which lets us apply Lemma~\ref{lem:local_integral_ineq} in full force and to obtain the stated estimates. 
	
	We set $\alpha = \min(n\eps/(n+\eps), n/K)$ and choose $\alpha' \in (0, \alpha)$.
	Applying Lemmas~\ref{lem:local_integral_ineq} and \ref{lem:diff_ineq_solution_r} we conclude that
	\begin{equation}\label{eq:ball_decay_for_alpha}
		\int_{B_r} \abs{Df}^n \leq C r^{\alpha'}
	\end{equation}
	for all $B_r = B^n(y, r) \subset B^n(x, 4R)$.  Therefore, it follows from Lemma \ref{lem:Holder_continuity_general_log} that $f$ is $(\alpha'/n)$-H\"older continuous in $B^n(x, R)$. Since continuity is a local property, we conclude that $f$ is continuous, and in particular bounded in $B^n(x, 4R)$. 
	
	Now, knowing that $f$ is locally bounded we may and do take $\delta=\eps$ in  Lemma~\ref{lem:local_integral_ineq}. Combining this with Lemma~\ref{lem:diff_ineq_solution_r}, we obtain the following decay estimate for the differential: 
	\begin{equation}\label{eq:ball_decay_for_alpha_log}
		\int_{B_r} \abs{Df}^n \leq C r^\alpha \log^\beta
		\Bigl(\frac{(4R)e^{1+K}}{r}\Bigr)\, , \qquad \textnormal{where} \quad  \begin{cases} \beta =0 & \textnormal{if } \;  \frac{n \eps }{n+\eps }\not= \frac{n}{K} \\
		\beta =1 \; \;  & \textnormal{if } \;  \frac{n \eps }{n+\eps }= \frac{n}{K}\, . 
		\end{cases}
	\end{equation}
	for all $B_r$. Thus, the desired H\"older continuity estimates for $f$ follow from Lemma~\ref{lem:Holder_continuity_general_log}.

\end{proof}

\section{Sharpness of the H\"older exponents}
Having Lemma \ref{lem:Holder_continuity}, the remaining part of proving Theorem \ref{thm:Holder} is to construct solutions which show that the obtained H\"older exponents cannot be improved. Recalling the notation $\gamma_K = K^{-1}$ and $\gamma_\eps = \eps / (n + \eps)$, we consider three different cases: $\gamma_K < \gamma_\eps$, $\gamma_K > \gamma_\eps$ and $\gamma_K = \gamma_\eps $.

For the first case $\gamma_K < \gamma_\eps$, we can simply use the standard radial example
\[
	f(x) = \abs{x}^{\frac{1}{K}} \frac{x}{\abs{x}}.
\]
Indeed, the mapping $f$ is $K$-quasiregular and hence satisfies \eqref{eq:main_definition} with $\sigma \equiv 0$, and we also have $f \notin C^{0, \gamma}(B^n(0, 1))$ for every $\gamma > K^{-1}$. 

Next, we discuss the case $\gamma_K>  \gamma_\eps$ in-depth.

\begin{ex}\label{ex:sobolev_dominant_case}
	Let $K \geq 1$ and $\eps > 0$ such that $K^{-1} > \eps/(n+\eps)$. We define a mapping $f \colon B^n(0, 1) \to \R^n$ with only a single non-vanishing coordinate function, namely
	\[
		f(x) = \left(1 + \abs{x}^\frac{\eps}{n+\eps}
			\log^{-\frac{1}{n}}\left( \frac{e}{\abs{x}}\right), 
			0, 0, \dots, 0 \right).
	\]
	This mapping lies in $W^{1, n}(B^n(0, 1), \R^n)$, with
	\[
		\nabla f_1(x) = \abs{x}^{-\frac{n}{n+\eps}} \left(
			\frac{\eps}{n+\eps} 
				\log^{-\frac{1}{n}} \left(\frac{e}{\abs{x}}\right)
			+ \frac{1}{n}
				\log^{-\frac{n+1}{n}} \left(\frac{e}{\abs{x}}\right)
		\right)\frac{x}{\abs{x}} \, . 
	\] 
	Furthermore, $J_f \equiv 0$.  Hence, the heterogeneous distortion inequality \eqref{eq:main_definition} for $f$ reduces to
	\[
		\abs{\nabla f_1}^n \leq \abs{f}^n \sigma^n.
	\]
	Since  $\abs{f(x)} = f_1(x) \geq 1$ for every  $x\in B^n(0, 1)$, the mapping $f$ solves the heterogeneous distortion inequality for  any $\sigma \geq \abs{\nabla f_1}$. We choose
	\[
		\sigma = \abs{\nabla f_1} = \abs{x}^{-\frac{n}{n+\eps}} \left(
			\frac{\eps}{n+\eps} 
				\log^{-\frac{1}{n}} \left(\frac{e}{\abs{x}}\right)
			+ \frac{1}{n}
				\log^{-\frac{n+1}{n}} \left(\frac{e}{\abs{x}}\right)
		\right)
	\]
	and then observe that
	\[
		\sigma^{n+\eps} \leq 2^{n+\eps} \abs{x}^{-n} \left(
			\frac{\eps}{n+\eps} 
				\log^{-\frac{n+\eps}{n}} \left(\frac{e}{\abs{x}}\right)
			+ \frac{1}{n}
			\log^{-\frac{(n+1)(n+\eps)}{n}}\left(\frac{e}{\abs{x}}\right)
		\right).
	\]
	We recall that for any $p > 1$, the function $\abs{x}^{-n} \log^{-p}(e/\abs{x})$ is integrable over $B^n(0, 1)$. Indeed,
	\begin{multline*}
		\int_{B^n(0, 1)} \abs{x}^{-n} \log^{-p}\frac{e}{\abs{x}}
		= C_n \int_0^1	\frac{r^{n-1} \dd r}{r^n \log^p \frac{e}{r}}\\
		= C_n \int_0^1 \frac{\dd}{\dd r} \left(\log^{-(p-1)} \frac{e}{r} \right) \dd r
		< \infty.
	\end{multline*}
	Hence,  the mapping $f$ solves~\eqref{eq:main_definition} with $\sigma \in L^{n+\eps}(B^n(0, 1))$. However, for any exponent $\gamma > \eps/(n+\eps)$, the map $f$ fails to be $\gamma$-Hölder continuous at the origin.
\end{ex}

The remaining part of the proof of Theorem \ref{thm:Holder} is therefore to provide an example in the special case $\gamma_K = \gamma_\eps$.

\begin{ex}\label{ex:magic_exponent}
	Let $K \geq 1$ and $\eps > 0$, and suppose that $K^{-1} = \eps/(n+\eps)$. We define a mapping $f \colon B^n(0, 1) \to \R^n$ by  
	\[
		f(x) = \abs{x}^\frac{1}{K} 
			\log^\frac{1}{2nK} \left(\frac{e}{\abs{x}}\right)
			\frac{x}{\abs{x}} + (2, 0, 0, \dots, 0).
	\]
	The mapping $f$ is hence obtained by shifting a   radially symmetric map of the form $(\Phi(\abs{x})/\abs{x}) x$, where $\Phi(t) = t^{1/K} \log^{1/(2nK)} (e/t)$. For $x \in B^n(0, 1)$ we have
	\begin{align*}
		\abs{\frac{\Phi(\abs{x})}{\abs{x}}}
			&= \abs{x}^{-\frac{K-1}{K}} 
				\log^{\frac{1}{2nK}} \left( \frac{e}{\abs{x}} \right) 
				\quad \text{and}\\
		\abs{\Phi'(\abs{x})}
			&= \abs{x}^{-\frac{K-1}{K}} \left(
				\frac{1}{K} \log^\frac{1}{2nK} 
					\left( \frac{e}{\abs{x}} \right)
				- \frac{1}{2nK} \log^{-\frac{2nK -1}{2nK}} 
				\left( \frac{e}{\abs{x}} \right)
			\right).
	\end{align*}
	Using these and the fact that $f$ is orientation preserving, we conclude, see e.g.~\cite[6.5.1]{Iwaniec-Martin_book}, that
	\begin{align*}
		\abs{Df(x)}^n
		&=  \max \left( \abs{\frac{\Phi(\abs{x})}{\abs{x}}}^n, \abs{\Phi'(\abs{x})}^n \right)\\
		&= \abs{x}^{-\frac{n(K-1)}{K}} 
		\log^{\frac{1}{2K}} \left( \frac{e}{\abs{x}} \right)
	\end{align*}
	and
	\begin{align*}
		KJ_f(x)
		&=  K\abs{\frac{\Phi(\abs{x})}{\abs{x}}}^{n-1}\abs{\Phi'(\abs{x})}\\
		&= \abs{x}^{\frac{-n(K-1)}{K}} 
		\left(\log^{\frac{n}{2nK}} \left( \frac{e}{\abs{x}} \right) 
			- \frac{1}{2n} \log^{\frac{n-1}{2nK} - \frac{2nK -1}{2nK}} 
			\left( \frac{e}{\abs{x}} \right) 
		\right)\\
		&= \abs{x}^{\frac{-n(K-1)}{K}} 
		\left(\log^{\frac{1}{2K}} \left( \frac{e}{\abs{x}} \right) 
			- \frac{1}{2n} \log^{- \frac{2nK -n}{2nK}} 
			\left( \frac{e}{\abs{x}} \right) 
		\right).
	\end{align*}
	
	Since $\Phi$ is increasing on $[0, 1]$, we have $\abs{f(x)} \geq 2-\Phi(1) = 1$ for all $x \in B^n(0, 1)$. Therefore, the heterogeneous distortion inequality \eqref{eq:main_definition} is satisfied if we choose
	\[
		\sigma^n(x) = \frac{1}{2n} \abs{x}^{-\frac{n(K-1)}{K}} 
			\log^{-\frac{2nK -n}{2nK}} \left( \frac{e}{\abs{x}} \right).
	\]
	We then observe that
	\[
		\sigma^{n+\eps}(x) = \left( \sigma^n(x) \right)^\frac{K}{K-1}
		= \frac{1}{(2n)^\frac{K-1}{K}} \abs{x}^{-n}
			\log^{-\frac{(2nK - n)K}{2nK(K-1)}} \left( \frac{e}{\abs{x}} \right),
	\]
	and since
	\[
		\frac{(2nK - n)K}{2nK(K-1)} = \frac{2nK^2 - nK}{2nK^2 - 2nK} > 1,
	\]
	we conclude that $\sigma \in L^{n+\eps}(B^n(0, 1))$. However, we have
	\[
		\frac{\abs{f(x) - f(0)}}{\abs{x - 0}^{\frac{1}{K}}}
		= \log^\frac{1}{2nK} \left( \frac{e}{\abs{x}} \right)
		\xrightarrow[x \to \infty]{} \infty,
	\]
	and therefore $f \notin C^{0, K^{-1}}_\loc(B^n(0, 1))$.
\end{ex}

The proof of Theorem \ref{thm:Holder} is thus complete.

\section{Sublevel sets and the logarithm}

In this section, we begin studying bounded entire functions $f$ satisfying \eqref{eq:main_definition}, with the goal of eventually reaching the Liouville type theorem stated in the Astala-Iwaniec-Martin question. Our main goal in this section is to show that if $f$ is not identically zero, then $\log \abs{f} \in W^{1,n}_\loc(\R^n)$. This is already notable, since this condition is not satisfied by all unbounded entire quasiregular maps. Our approach does not rely on the theory of  partial differential equations. Instead, the proof is based on two main tools: integration by parts and truncating $f$ with respect to its level sets.

\subsection{Global integrability} 

We begin with a simple global integrability result for $Df$ when $f$ is an entire mapping that solves the heterogeneous distortion inequality \eqref{eq:main_definition}.

\begin{lemma}\label{lem:Df_global_integrability}
	Suppose that $f \in W^{1,n}_\loc(\R^n, \R^n)$ satisfies the heterogeneous distortion inequality \eqref{eq:main_definition} with $K \in [1, \infty)$ and $\sigma \in L^{n}(\R^n)$. If $f$ is bounded, then $\abs{Df} \in L^{n}(\R^n)$.
\end{lemma} 
\begin{proof}
	Let $\eta_r \colon \R^n \to [0,1]$ be a smooth mapping chosen such that $\eta\vert_{B^n(0, r)} \equiv 1$, $\eta \vert_{\R^n \setminus B^n(0, 2r)} \equiv 0$, and $\abs{\nabla \eta} \leq 2/r$. Now, by using the heterogeneous distortion inequality, the Caccioppoli estimate of Lemma \ref{lem:caccioppoli}, and H\"older's inequality, we obtain
	\begin{align*}
		\int_{\R^n} \eta_r^n \abs{Df}^{n}
		&\leq K\int_{\R^n} \eta_r^n J_f 
			+ \int_{\R^n} \eta_r^n \abs{f}^n \sigma^n\\
		&\leq K \int_{\R^n} (\eta_r \abs{Df})^{n-1} \abs{f} 
				\abs{\nabla \eta_r}
			+ \norm{f}_\infty^n \int_{\R^n} \eta_r^n \sigma^n\\
		&\leq  
			\left( \int_{B^n(0, 2r)} \abs{f}^n  \abs{\nabla \eta_r}^n 
				\right)^{\frac{1}{n}}
			\left( \int_{\R^n} \eta_r^n \abs{Df}^{n} \right)^{\frac{n-1}{n}}
			+ \norm{f}_\infty^n \norm{\sigma}_n^n\\
		&\leq 4\omega_n\norm{f}_\infty 
			\left( \int_{\R^n} \eta_r^n \abs{Df}^{n} \right)^{\frac{n-1}{n}}
			+ \norm{f}_\infty^n \norm{\sigma}_n^n
	\end{align*}
	Hence, we obtain an upper bound on the integral of $\eta^n_r \abs{Df}^n$ independent on $r$. Letting $r \to \infty$ yields the claim.
\end{proof}

\subsection{Level set methods}

We just proved that  for a bounded  $f \colon \R^n \to \R^n$ satisfying \eqref{eq:main_definition} with $\sigma \in L^n$, the differential $\abs{Df}$ lies in $L^n (\R^n)$. Therefore,  by Lemma \ref{lem:zero_Jacobian}, the integral of $J_f$ over the entire space $\R^n$ is zero. We now proceed to improve this by showing that the integral of the Jacobian also vanishes over every strict sublevel set of $\abs{f}$.

\begin{lemma}\label{lem:level_set_zero_Jacobian}
	Let $f \in W^{1, n}_\loc(\R^n, \R^n)$. Suppose that $\abs{Df} \in L^n(\R^n)$. Then for every $t > 0$, we have
	\[
		\int_{\left\{x \in \R^n \colon \abs{f} < t \right\}} J_f = 0.
	\]
\end{lemma}
\begin{proof}
	Let $t> \eps > 0$, and let $\psi=\psi_{t, \eps} \colon [0, \infty) \to [0, \infty)$ be a non-decreasing smooth function such that $\psi \vert_{[0, t-\eps]} = \id$, $\psi \vert_{[t, \infty)} \equiv t$, and $\abs{\psi'} \leq 2$. Let $h_{t, \eps} \colon \R^n \to \R^n$ be the radial function defined by
	\[
		h_{t, \eps}(x) =  \psi_{t, \eps}(\abs{x}) \frac{ x}{\abs{x}} .
	\]
	Then $h_{t, \eps}$ is a smooth and 2-Lipschitz regular mapping. Consequently,  the chain rule applies, $Df_{t, \eps}(x) = Dh_{t, \eps}(f(x)) Df(x)$ for a.e.\ $x$, and $f_{t, \eps} = h_{t, \eps} \circ f$ lies in $W^{1, n}_\loc(\R^n, \R^n)$, see e.g.~\cite[p.130]{Evans-Gariepy-book}.
	
	In particular, since $\abs{Dh_{t, \eps}} \leq 2$, we have $Df_{t, \eps} \in L^n(\R^n)$. Therefore, Lemma \ref{lem:zero_Jacobian} yields that
	\[
		\int_{\R^n} (J_{h_{t, \eps}} \circ f) J_f
		= \int_{\R^n} J_{f_{t, \eps}} = 0.
	\]
	As $\eps \to 0$, we have $J_{h_{t, \eps}} \to \chi_{[0, t)}$ pointwise where  $\chi_E$ denotes the characteristic function $\chi_E$ of a set $E$.  Hence, the claim follows by letting $\eps \to 0$ and applying  the dominated convergence theorem for the Lebesgue integral. 
\end{proof}

Lemma \ref{lem:level_set_zero_Jacobian} is our main tool in showing that, for an entire non-identically zero solution $f$, the function $\abs{\nabla \log \abs{f}}$ belongs to  $L^n (\R^n)$. Towards this, we first prove that the function $\abs{Df}^n/\abs{f}^n$ is globally integrable.

\begin{lemma}\label{lem:log_gradient_int}
	Suppose that $f \in W^{1,n}_\loc(\R^n, \R^n)$ solves the heterogeneous distortion inequality \eqref{eq:main_definition} with $K \in [1, \infty)$ and $\sigma \in L^{n}(\R^n)$. If $f$ is  bounded, then $J_f/\abs{f}^n$ is integrable,
	\begin{align*}
		\int_{\R^n} \frac{J_f}{\abs{f}^n}
		&= 0,
		&&\text{and}&
		\int_{\R^n} \frac{\abs{Df}^n}{\abs{f}^n} &\leq
		\int_{\R^n} \sigma^n < \infty.
	\end{align*}
\end{lemma}

Here and in what follows, we interpret $J_f/\abs{f}^n = 0$ when $J_f = 0$, and similarly, $\abs{Df}^n/\abs{f}^n = 0$ when $\abs{Df} = 0$.

\begin{proof}
	We split $J_f$ into its positive and negative parts $J_f = J_f^+ - J_f^-$. These parts satisfy the inequality
	\begin{equation}\label{eq:hdirestatement}
		\abs{Df}^n + KJ_f^- \leq K J_f^+ + \sigma^n \abs{f}^n.
	\end{equation}
	In particular, we have 
	\begin{equation}\label{eq:J_festimate}
		J_f^- \leq K^{-1} \abs{f}^n \sigma^n.
	\end{equation}
	Indeed, this is trivial when $J_f \ge 0$, and if $J_f <0$, then~\eqref{eq:J_festimate} follows from~\eqref{eq:hdirestatement}. 
	Hence, $J_f^- = 0$ a.e.\ where $\abs{f} = 0$, and we have
	\begin{equation}\label{eq:J_f-integrable}
	\int_{\R^n} \frac{J_f^-}{\abs{f}^n} 
	\leq \frac{1}{K} \int_{\R^n} \sigma^n < \infty \, . 
	\end{equation}
	
	For every $t>0$,  we denote the strict sublevel set of $\abs{f}$ at $t$ by $L_t= \left\{x \in \R^n \colon \abs{f} <t\right\}$.
	By Lemmas \ref{lem:Df_global_integrability} and \ref{lem:level_set_zero_Jacobian}, we have
	\[
		\int_{L_t} J_f = 0 \qquad \textnormal{ for every } t > 0 \, . 
	\]
	In particular, we have for every $t > 0$ that
	\[
		\int_{L_t} J_f^+
		= \int_{L_t} J_f^- \, .
	\]
	Multiplying  this estimate by $t^{-n-1}/n$, we obtain
	\begin{equation}\label{eq:the_ineq_you_integrate}
		\int_{\R^n} \frac{t^{-n-1}}{n} J_f^+ 
	 		\chi_{\left\{x \in \R^n \colon \abs{f} < t \right\}}
	 	= \int_{\R^n} 
	 		\frac{t^{-n-1}}{n} J_f^-
	 		\chi_{\left\{x \in \R^n \colon \abs{f} < t \right\}}.
	\end{equation}
	By integrating \eqref{eq:the_ineq_you_integrate} over $(0, \infty)$ with respect to $t$ and applying the Fubini--Tonelli theorem to change the order of integration, we have
	\[
		\int_{\R^n} J_f^+ 
			\int_{\abs{f(x)}}^{\infty} \frac{t^{-n-1}}{n} \dd t \dd x
		= \int_{\R^n} J_f^-
			\int_{\abs{f(x)}}^{\infty} \frac{t^{-n-1}}{n} \dd t \dd x.
	\]
	Evaluating the inner integral yields
	\[
		\int_{\R^n} \frac{J_f^+}{\abs{f}^n}
		= \int_{\R^n} \frac{J_f^-}{\abs{f}^n}
		< \infty
	\]
	where the finiteness of the integrals follows from~\eqref{eq:J_f-integrable}.
	
	We therefore conclude that $J_f^+ = 0$ a.e. where $\abs{f} = 0$, that $J_f^+/\abs{f}^n$ and therefore also $J_f/\abs{f}^n$ are integrable, and that
	\[
		\int_{\R^n} \frac{J_f}{\abs{f}^n} = 0.
	\]
	Finally, by the heterogeneous distortion inequality \eqref{eq:main_definition}, we see that $\abs{Df} = 0$ a.e.\ where $\abs{f} = 0$, and that
	\[
		\int_{\R^n} \frac{\abs{Df}^n}{\abs{f}^n}
		\leq K{\int_{\R^n} \frac{J_f}{\abs{f}^n}}
		+ \int_{\R^n} \sigma^n
		= \int_{\R^n} \sigma^n.
	\]
\end{proof}

\subsection{The logarithm}

With the global integrability of $\abs{Df}^n/\abs{f}^n$ shown, we now proceed to study the Sobolev regularity of $\log \abs{f}$. 

Let $B_R=B^n(x,R)$ be a ball in $\R^n$ with $R>0$. Suppose that $f \in L^\infty(B_R, \R^n) \cap W_{\loc}^{1,n}(B_R, \R^n)$, and that $\abs{Df}/\abs{f} \in L^n(B_R)$.  For every $\lambda > 0$, we denote by $\abs{f}_\lambda$ the function $x \mapsto \max(\abs{f}, \lambda)$. We then proceed to study the functions $ \log \abs{f}_\lambda$. Since $f$ is bounded and the function $h_\lambda \colon \R^n \to \R$ given by $h_\lambda (x) = \log \max(\abs{x}, \lambda )$ is locally Lipschitz, we may use the chain rule of Lipschitz and Sobolev maps to obtain that $\log \abs{f}_\lambda = h_\lambda \circ f \in W^{1,n}_\loc(B_R)$; see e.g.~\cite[Theorem 2.1.11]{Ziemer-book}. Moreover, we have the uniform estimate
\begin{equation}\label{eq:trunc_log_unif_bound}
	\abs{\nabla \log \abs{f}_\lambda}^n 
	= \frac{\abs{\nabla \abs{f}}^n}{\abs{f}^n} 
		\chi_{\left\{ \abs{f} > \lambda \right\}}
	\leq \frac{\abs{Df}^n}{\abs{f}_\lambda^n}
	\leq \frac{\abs{Df}^n}{\abs{f}^n}
	< \infty
\end{equation}
which is independent of $\lambda$.

By using these truncated logarithms as a tool, we achieve the following result.

\begin{lemma}\label{lem:log_f_is_Sobolev}
Let  $f \colon B_R \to \R^n$ be a bounded and not identically zero mapping.	Suppose that $f \in W^{1,n}_\loc(B_R, \R^n)$  and that $\abs{Df}/\abs{f} \in L^n(B_R).$   Then $f^{-1}\{0\}$ has zero Lebesgue measure, the measurable function $\log \abs{f}$ lies in $W^{1,n}_\loc(B_R)$, and
	\[
		\abs{\nabla \log \abs{f}} \leq \frac{\abs{Df}}{\abs{f}} \in L^n(B_R).
	\]
\end{lemma}
\begin{proof}
	By our assumptions, the set $\abs{f}^{-1} (0, \infty)$ has positive measure. Hence, there exists $t \in (0, 1)$ such that $F_t = \left\{ x \in B_R : t^{-1} > \abs{f(x)} > t \right\}$ has positive measure. For every $\lambda > 0$, we denote by $f_\lambda \colon B_R \to \R$ the function $f_\lambda = \log \abs{f}_\lambda$.
	
	Our first goal is to show that $\log \abs{f} \in L^1(B_R)$. For the proof, we assume towards a contradiction that the integral of $\abs{\log \abs{f}}$ over $B_R$ is instead infinite. In this case, since the functions $f_\lambda$ are uniformly bounded from above and decrease to $\log \abs{f}$ monotonically as $\lambda \to 0+$, we have $\lim_{\lambda \to 0+} (f_\lambda)_{B_R} = -\infty$; recall that $(f_\lambda)_{B_R}$ stands for the integral average value of the function $f_\lambda$ over $B_R$.
	
	By the Sobolev-Poincar\'e inequality and \eqref{eq:trunc_log_unif_bound}, we have the upper bound
	\[
		\frac{1}{m_n(B_{R})} \int_{B_{R}} \abs{f_\lambda - (f_\lambda)_{B_{R}}}
		\leq C_n \left( \int_{B_{R}} \abs{\nabla f_\lambda}^n \right)^\frac{1}{n}
		\leq C_n \left( \int_{B_R} \frac{\abs{Df}^n}{\abs{f}^n} \right)^\frac{1}{n}.
	\]
	This upper bound,  independent of $\lambda$,  is finite by our assumptions.  We also have the lower bound
	\[
		\frac{1}{m_n(B_{R})} \int_{B_{R}} \abs{f_\lambda - (f_\lambda)_{B_{R}}}
		\geq \frac{m_n(F_t)}{m_n(B_{R})} \left( \abs{(f_\lambda)_{B_{R}}}  - \log t^{-1}\right).
	\]
	Since $\lim_{\lambda \to 0+} (f_\lambda)_B = -\infty$, we arrive at a contradiction. Hence, $\log \abs{f} \in L^1(B_R)$. In particular, it follows that $\log \abs{f}$ is finite almost everywhere, and therefore $f^{-1}\{0\}$ has zero Lebesgue measure.
	
	Now, for $\lambda < 1$, we have
	\[
		\int_{B_{R}} \abs{\log \abs{f} - f_\lambda}
		\leq \int_{f^{-1} [0, \lambda)} \abs{\log \abs{f}} \to 0 \qquad \textnormal{when } \lambda \to 0+ \, ,
		\]
		and
	\[
		\int_{B_{R}} \abs{\nabla f_\lambda - \frac{\nabla \abs{f}}{\abs{f}}}^n
		\leq \int_{f^{-1} [0, \lambda]} \frac{\abs{Df}^n}{\abs{f}^n}
		 \to 0 \qquad \textnormal{when } \lambda \to 0+ \, . 
	\]
	Therefore, $f_\lambda \to \log \abs{f}$ in $L^1 (B_R)$ and  $\nabla f_\lambda \to (\nabla \abs{f})/\abs{f}$ in $L^n(B_R)$.  Thus, the weak gradient of $\log \abs{f}$ equals  $(\nabla \abs{f})/\abs{f}$. Since $(\nabla \abs{f})/\abs{f} \in L^n(B_R, \R^n)$, the Sobolev embedding theorem shows that $\log \abs{f} \in L^{n}_\loc(B_R)$, and hence $\log \abs{f} \in W^{1,n}_\loc(B_R)$.
\end{proof}

\section{Non-existence of zeroes}

In this section  we will  show that $\log \abs{f}$ is locally H\"older continuous if $f$ is a bounded entire solution  to the heterogeneous distortion inequality with $\sigma \in L^n(\R^n) \cap L^{n+\eps}_{\loc} (\R^n)$. This will prove Theorem \ref{thm:nonzero}.
 Our approach again  mimics the lines of reasoning by Morrey and  is based on obtaining a quantitative integral estimate for $\abs{Df}^n/\abs{f}^n$ over balls. This is done by employing a suitable isoperimetric inequality.  
 \subsection{Logarithmic isoperimetric inequality} We then proceed to show the following isoperimetric-type estimate for $\abs{Df}^n/\abs{f}^n$. As before, we use $B_r = B^n(x, r)$ to denote a ball in $\R^n$ around a fixed point $x$.
 
\begin{lemma}\label{lem:logisoperimetric}
Let $f \in L^\infty_\loc(B_R, \R^n) \cap W_{\loc }^{1,n} (B_R, \R^n)$, where $R > 0$. If $\abs{Df}/\abs{f} \in L_{\loc}^n (B_R)$, then there is a constant $C_n$ such that for a.e. $r\in (0,R)$, we have
\begin{equation}\label{eq:log_isoperimetric_n}
\int_{B_r} \frac{J_f}{\abs{f}^n} \le C_n \, r  \int_{\partial B_r} \frac{\abs{Df}^{n}}{\abs{f}^{n} } \, .  
\end{equation}

\end{lemma}
The main idea behind the proof is to write
\[  \frac{J_f(x)\,}{\abs{f(x)}^n} \vol_n  = \dd \omega,\]
where $\omega$ is a certain differential $(n-1)$-form, and then to use Stokes' theorem. This method actually gives similar estimates for integrals of the more general form $\psi (\abs{f}) J_f $ over  balls. The precise estimate obtained is given by the following lemma, which is a variant of \cite[Lemma 2.1]{Onninen-Zhong_MFD-loglog-proof} by Onninen and Zhong. We provide a proof here due to our assumptions being slightly weaker than in~\cite{Onninen-Zhong_MFD-loglog-proof}.

\begin{lemma}\label{lem:Onninen_Zhong_lemma_variation}
	Let $\Omega \subset \R^n$ be a domain. Suppose that $f \colon \Omega \to \R^n$ is in $L^\infty_\loc(\Omega, \R^n ) \cap W^{1,n}_\loc(\Omega, \R^n)$.  If $\Psi \colon [0, \infty) \to \R$ is a piecewise $C^1$-smooth function with $\Psi'$ locally bounded,  then for every test function $\eta \in C^\infty_0(\Omega)$, we have
	\[
		\abs{\int_{\Omega}  \eta \bigl[n \Psi(|f|^2) +
			 2 \abs{f}^2 \Psi'(\abs{f}^2)\bigr]J_f}
		\leq \sqrt{n} \int_{\Omega}  \abs{\nabla \eta} \abs{f} \Psi(\abs{f}^2)\abs{Df}^{n-1}.
	\]
\end{lemma}
Before jumping into the proof, we comment on the well-definedness of the left integrand. The function $\Psi'$ is defined outside finitely many jump points $y \in [0, \infty)$. Consider the set $A_y = \{x \in \Omega : \abs{f(x)}^2 = y\}$. Then a.e.\ on $A_y$, we have $\nabla \abs{f}^2 = 0$; see e.g.\ \cite[Corollary 1.21]{Heinonen-Kilpelainen-Martio_book}. Since $\nabla \abs{f}^2 = \sum_{i=1}^n 2f_i \nabla f_i$, this implies that for a.e.\ $x \in A_y$, we have $\abs{f(x)} = 0$ or some non-zero linear combination of $\{(df_i)_x : i = 1, \dots, n \}$ vanishes. In the latter case, $J_f(x) = 0$. Hence, almost everywhere where $\Psi'(\abs{f}^2)$ is not defined, we have $\abs{f} = 0$ or $J_f = 0$, making the integrand in the statement well defined.
\begin{proof}[Proof of Lemma~\ref{lem:Onninen_Zhong_lemma_variation}]
	By switching to a smaller domain $\Omega$ which still contains the support of $\eta$, we may assume that $\Omega$ is bounded and $f \in W^{1,n}(\Omega, \R^n) \cap L^\infty(\Omega, \R^n)$. By boundedness of $\abs{f}$ and the local boundedness of $\Psi$ and $\Psi'$, we may also assume that $\Psi$ and $\Psi'$ are bounded. 
	
	We consider the function
	\[
		F_i = (f_1, \dots, f_{i-1}, \eta \Psi(\abs{f}^2)f_i, 
			f_{i+1}, \dots, f_n).
	\]
	Since $\Psi'$ is bounded and $\abs{f}^2$ is Sobolev, the chain rule of Lipschitz and Sobolev maps yields that $\nabla(\Psi(\abs{f}^2)) = \Psi'(\abs{f}^2) \nabla \abs{f}^2 = 2 \Psi'(\abs{f}^2) \sum_{j=1}^n f_j \nabla f_j$ a.e.\ on $\Omega$, see e.g.~\cite[Theorem 2.1.11]{Ziemer-book}. By further using the product rules of Sobolev mappings, we see that $\eta \Psi(\abs{f}^2)f_i$ has a locally integrable weak gradient given by 
	\begin{multline*}
		\nabla(\eta \Psi(\abs{f}^2)f_i)\\
		= \Psi(\abs{f}^2)f_i \nabla \eta 
			+ 2 \eta \Psi'(\abs{f}^2) f_i 
				\left(\sum_{j=1}^n f_j \nabla f_j\right)
			+ \eta \Psi(\abs{f}^2)\nabla f_i.
	\end{multline*}
	Using the fact that $\eta$, $\Psi(\abs{f}^2)$, $\Psi'(\abs{f}^2)$ and $f_i$ are bounded, we then conclude that this weak gradient is in $L^n(\Omega, \R^n)$.
	
	We therefore have that $\eta^n \Psi(\abs{f}^2)f_i \in W^{1, n}(\Omega)$. Consequently, $F_i \in W^{1, n}(\Omega)$, and therefore $J_{F_i}$ is integrable. Since $F_i$ also has a compactly supported coordinate function, we therefore have
	\[
		\int_{\Omega} J_{F_i} = 0.
	\]
	By writing $J_{F_i} \vol_n$ as a wedge product, we obtain
	\begin{multline*}
		\int_{\Omega}  \eta \bigl[\Psi(|f|^2) +
			2 f_i^2 \Psi'(\abs{f}^2)\bigr]J_f \vol_n\\
		= -\int_{\Omega}  \Psi(\abs{f}^2) f_i df_1 \wedge \dots \wedge df_{i-1} \wedge d\eta \wedge df_{i+1} \wedge \dots \wedge df_n.
	\end{multline*}
	By summing over $i$, and by using the fact that $\abs{\alpha_1 \wedge \dots \wedge \alpha_n} \leq \abs{\alpha_1} \cdots \abs{\alpha_n}$ for 1-forms $\alpha_1, \dots, \alpha_n$, the claim follows.
\end{proof}

With the proof of Lemma \ref{lem:Onninen_Zhong_lemma_variation} complete, we then proceed to prove Lemma \ref{lem:logisoperimetric}.
\begin{proof}[Proof of Lemma~\ref{lem:logisoperimetric}]
	We first prove an isoperimetric estimate of the following form: for a.e. $r\in (0,R)$ and every constant $c$, we have
	\begin{equation}\label{eq:isoperimetric_apu}
		\int_{B_r} \frac{J_f}{\abs{f}^n} \le C_n \int_{\partial B_r} \frac{\abs{Df}^{n-1}}{\abs{f}^{n-1}} \bigl|\log \abs{f} - c\bigr| \, . 
	\end{equation}
	
	Hence, fix a $c \in \R$ and let $r \in (0, R)$. For all sufficiently large $j \in \N$, we select cutoff functions $\eta_j \in C^\infty_0 (B_r)$ such that $\eta_j \leq \eta_{j+1} \le 1$, $\eta_j (x)=1$ for all $x\in B_{r-1/j}$, and $\sup \{ \abs{\nabla \eta_j (x)} \colon x \in B_r \} \le 1/j$. We also fix $a > 0$ and $\eps \in (0, 1)$, and define a function $\Psi_{a, \eps} \colon [0, \infty) \to \R$ by
	\[
		\Psi_{a, \eps}(t) = \begin{cases}
			t^{-\frac{n}{2}}\left(\frac{1}{2}\log (t+\eps) - c\right), 	
				&t \geq a^2\\
			a^{-n}(\log \sqrt{a^2+\eps} - c), 
				&t \leq a^2
		\end{cases}.
	\]
	The function $\Psi_{a, \eps}$ is piecewise $C^1$ and its derivative is locally bounded. Moreover, we have
	\[
		n \Psi_{a, \eps}(t^2) + 2 t^2 \Psi_{a, \eps}'(t^2)
		= \begin{cases}
			t^{-(n-2)}(t^2 + \eps)^{-1} & t > a\\
			n a^n (\log \sqrt{a^2+\eps} - c), & t < a.\\
		\end{cases}
	\]
	
	Hence, by using Lemma \ref{lem:Onninen_Zhong_lemma_variation} with $\Psi=\Psi_{a, \eps}$ and $\eta=\eta_j$, we obtain that
	\begin{align*}
		\abs{\int_{\left\{\abs{f} > a \right\}}
			\frac{\eta_j J_f}{\abs{f}^{n-2}(\abs{f}^2 + \eps)}}
		& \leq \frac{C_n}{j} \int^r_{r-\frac{1}{j}} \int_{\partial B_s} \frac{ \abs{Df}^{n-1} \abs{f}}{\abs{f}_a^{n}}
						\abs{\log \sqrt{\abs{f}_a^2 + \eps} - c} \, d s\\
		&\qquad+\abs{\int_{\left\{\abs{f} < a \right\}}
			\frac{n\eta_j J_f(\log \sqrt{a^2 + \eps} - c)}{a^n}},
			\end{align*}
	where we recall that $\abs{f}_a= \max(\abs{f}, a)$. We then let $j \to \infty$, where we use monotone convergence and the Lebesgue differentiation theorem to obtain
	\begin{multline*}
		\abs{\int_{\left\{\abs{f} > a \right\}  \cap B_r }
			\frac{J_f}{\abs{f}^{n-2}(\abs{f}^2 + \eps)}}
		\leq C_n  \int_{\partial B_r} \frac{ \abs{Df}^{n-1} \abs{f}}{\abs{f}_a^{n}}
						\abs{\log \sqrt{\abs{f}_a^2 + \eps} - c} \\
		+  \abs{\int_{\left\{\abs{f} < a \right\} \cap B_r }
			\frac{n J_f(\log \sqrt{a^2 + \eps} - c)}{a^n}}
	\end{multline*}
	 for a.e. $r\in (0,R)$. Combining this with Hadamard's inequality $\abs{J_f} \leq \abs{Df}^n$ yields
	\begin{align}
		\nonumber
		\abs{\int_{\left\{\abs{f} > a \right\}  \cap B_r }
			\frac{J_f}{\abs{f}^{n-2}(\abs{f}^2 + \eps)}}
		&\leq C_n  \int_{\partial B_r} \frac{ \abs{Df}^{n-1} \abs{f}}{\abs{f}_a^{n}}
						\abs{\log \sqrt{\abs{f}_a^2 + \eps} - c}  \\
		\label{eq:isoperimetric_apu_apu}
		&\quad+ n \abs{\log \sqrt{a^2 + \eps} - c} {\int_{\left\{\abs{f} < a \right\} \cap B_r }
			\frac{\abs{Df}^n}{\abs{f}^n } } \, . 
	\end{align}
	for a.e. $r\in (0,R)$.
	
	Next, we let $a \to 0+$. Since $\abs{Df}^n/\abs{f}^n$ is integrable and $f^{-1}\{0\}$ has   zero measure by Lemma \ref{lem:log_f_is_Sobolev}, the last integral in~\eqref{eq:isoperimetric_apu_apu} goes to zero as $a \to 0+$. For the first integral on the right hand side of~\ref{eq:isoperimetric_apu_apu}, we observe that its integrand is dominated by the function  $(\abs{Df}^{n-1}/\abs{f}^{n-1}) (\log\abs{f} + C)$ for some $C>0$. This dominant is in $L^1_\loc(B_R)$ for any $C > 0$, since $\abs{Df}^{n-1}/\abs{f}^{n-1} \in L^{n/(n-1)}(B_R)$, and since $\log \abs{f} \in L^n_\loc(B_R)$ by Lemma \ref{lem:log_f_is_Sobolev}. Consequently, the dominant is also in $L^1(\partial B_r)$ for a.e.\ $r \in (0, R)$ by the Fubini-Tonelli theorem. Hence, we may apply the dominated convergence theorem as $a \to 0+$ in \eqref{eq:isoperimetric_apu_apu}, and therefore obtain
	\[ \abs{\int_{B_r }
			\frac{J_f}{\abs{f}^{n-2}(\abs{f}^2 + \eps)}}
		 \leq C_n  \int_{\partial B_r} \frac{ \abs{Df}^{n-1} }{\abs{f}^{n-1}}
						\abs{\log \sqrt{\abs{f}^2 + \eps} - c} \, 
	\]
	for a.e.\ $r \in (0, R)$.
	
	We then let $\eps \to 0^+$ and again use  the dominated convergence theorem, obtaining the claimed inequality~\eqref{eq:isoperimetric_apu} for a.e.\ $r \in (0, R)$ and our fixed value of $c$. Consequently, if $S \subset \R$ is a countable dense subset, then~\eqref{eq:isoperimetric_apu} holds for a.e.\ $r \in (0, R)$ and all $c \in S$. We then obtain~\eqref{eq:isoperimetric_apu} for all $c \in \R$ and a.e.\ $r \in (0, R)$ by taking limits, since the constant in \eqref{eq:isoperimetric_apu} is independent of $c$ and since $\abs{Df}^{n-1} / \abs{f}^{n-1}$ is integrable over $\partial B_r$ for a.e.\ $r \in (0, R)$.

	It remains to derive the statement of the lemma from~\eqref{eq:isoperimetric_apu}. We denote 
	\[
		\osc (\log \abs{f}, \partial B_r) 
		= \underset{\partial B_r}{\sup} \log \abs{f} 
			- \underset{\partial B_r}{\inf} \log \abs{f}\, .  
	\]
	Since $\abs{\nabla \log \abs{f}} \in L^n (B_R)$, the \emph{Sobolev embedding theorem on spheres}~\cite[Lemma 2.19]{Hencl-Koskela-book} implies that, after changing $f$ in a set of measure zero, we have
	\begin{equation}\label{eq:sob_emb_on_spheres}
		\osc(\log \abs{f}, \partial B_r)  
		\le C_n r^\frac{1}{n} \norm{\nabla \log \abs{f}}_n
		\le C_n r^\frac{1}{n} 
			\left(  \int_{\partial B_r}  
				\frac{\abs{Df}^n}{\abs{f}^n} \right)^\frac{1}{n } 
	\end{equation}
	for a.e.\ $r \in (0, R)$. Moreover, if $r \in (0, R)$ is such that~\eqref{eq:isoperimetric_apu} is valid, we may select $b \in \partial B_r$ and take $c = \log \abs{f(b)}$, in which case~\eqref{eq:isoperimetric_apu} yields
	\begin{equation}\label{eq:isoperimetric_osc}
		\int_{B_r} \frac{J_f}{\abs{f}^n} 
		\le C_n  \osc(\log \abs{f}, \partial B_r)   
		\int_{\partial B_r} \frac{\abs{Df}^{n-1}}{\abs{f}^{n-1}}  \, . 
	\end{equation}
	for a.e.\ $r \in (0, R)$.
	
	Now, combining \eqref{eq:sob_emb_on_spheres}, \eqref{eq:isoperimetric_osc} and H\"older's inequality, we obtain 
	\[
		\begin{split}
		\int_{B_r} \frac{J_f}{\abs{f}^n} & \le C_n   r^\frac{1}{n} \left(  \int_{\partial B_r}  \frac{\abs{Df}^n}{\abs{f}^n} \right)^\frac{1}{n }   \int_{\partial B_r} \frac{\abs{Df}^{n-1}}{\abs{f}^{n-1}} \\
		& \le  C_n   r^\frac{1}{n} \left(  \int_{\partial B_r}  \frac{\abs{Df}^n}{\abs{f}^n} \right)^\frac{1}{n }   r^\frac{n-1}{n} \left(  \int_{\partial B_r}  \frac{\abs{Df}^n}{\abs{f}^n} \right)^\frac{n-1}{n } \\
		& = C_n  r  \int_{\partial B_r}  \frac{\abs{Df}^n}{\abs{f}^n}  \, . 
		\end{split}
	\]
	This concludes the proof of Lemma~\ref{lem:logisoperimetric}.
	\end{proof}
\subsection{H\"older continuity of the logarithmic function}
Here we  complete the proof of Theorem \ref{thm:nonzero}. We recall the statement first.

\begin{customthm}{\ref{thm:nonzero}}
	Suppose that $f \in W^{1,n}_\loc(\R^n, \R^n)$ satisfies the heterogeneous distortion inequality \eqref{eq:main_definition} with $K \in [1, \infty)$ and $\sigma \in L^{n}(\R^n) \cap L^{n+\eps}_\loc(\R^n)$ where $\eps > 0$. If $f$ is bounded and  $f \not \equiv 0$, then $0 \notin f(\R^n)$.
\end{customthm}	
The proof is based on the following logarithmic counterpart of Lemma \ref{lem:local_integral_ineq}, where the use of the isoperimetric inequality is replaced with Lemma \ref{lem:logisoperimetric}.

\begin{lemma}\label{lem:log_gradient_local_estimate}
	Suppose that $f \colon \R^n \to \R^n$ is in $W^{1,n}_\loc(\R^n, \R^n)$ and solves the heterogeneous distortion inequality \eqref{eq:main_definition} with $K \in [1, \infty)$, and $\sigma \in L^n(\R^n)$. If $f$ is  bounded and not the constant function $f \equiv 0$, then for every $x \in \R^n$ and almost every ball $B_r = B^n(x,r) \subset \R^n$, we have
	\[
	\int_{B_r} \frac{\abs{Df}^n}{\abs{f}^n}
	\leq C_n(K) r \int_{\partial B_r} \frac{\abs{Df}^n}{\abs{f}^n}
	+ \int_{B_r} \sigma^n.
	\]
\end{lemma}
\begin{proof}
By Lemmas \eqref{lem:Df_global_integrability} and Lemma~\ref{lem:log_gradient_int}, we have $\abs{Df}/\abs{f} \in L^n(\R^n)$. Hence, the heterogeneous distortion inequality is in this case equivalent with the inequality
\begin{equation}\label{eq:heterogeneous_dist_divided}
	\frac{\abs{Df(x)}^n}{\abs{f(x)}^n} \le K \frac{J_f(x)}{\abs{f(x)}^n} + \sigma^n(x) \, \qquad \textnormal{for a.e. } x \in \R^n, 
\end{equation}
since $m_n (f^{-1} \{0\}) =0$ by Lemma~\ref{lem:log_f_is_Sobolev}. We then combine \eqref{eq:heterogeneous_dist_divided} with Lemma \ref{lem:logisoperimetric}, and the claim follows.
\end{proof}	

\begin{proof}[Proof of Theorem~\ref{thm:nonzero}]
	Let $B_R= B^n(0,R)$ for some $R > 0$. By Lemmas~\ref{lem:log_gradient_int} and \ref{lem:log_f_is_Sobolev}, we have that $\log \abs{f} \in W^{1,n}(\mathbb B_R)$, and we moreover have $\abs{\nabla \log \abs{f}} \leq \abs{Df}/\abs{f}$ almost everywhere.
	
	Let then $B_r = B^n(x, r) \subset B_R$. By Lemma \ref{lem:log_gradient_local_estimate} and H\"older's inequality, we have
	\begin{multline*}
		\int_{B_r} \frac{\abs{Df}^n}{\abs{f}^n}
		\leq C_n(K) r \int_{\partial B_r} \frac{\abs{Df}^n}{\abs{f}^n}
		+ \int_{Q_r} \sigma^n\\
		\leq C_n(K) r \int_{\partial B_r} \frac{\abs{Df}^n}{\abs{f}^n}
			+ (2r)^\frac{n\eps}{n + \eps} \left(
				\int_{\mathbb B_R} \sigma^{n+\eps} \right)^\frac{n}{n+\eps}.
	\end{multline*}
	Since this holds for every $x \in B_R$ and a.e.\ $r \in (0, \dist (x, \partial \mathbb B ))$, Lemma \ref{lem:diff_ineq_solution_r} yields that 
	\[
		\int_{B_r} \abs{\nabla \log \abs{f}}^n
		\leq \int_{B_r} \frac{\abs{Df}^n}{\abs{f}^n}
		\leq C r^\alpha,
	\]
	where $C$ and $\alpha$ are independent of our choice of $B_r \subset B_R$. Hence, by Lemma \ref{lem:Holder_continuity_general_log}, we have that $\log \abs{f}$ is H\"older continuous in $B_{R/4}$. Therefore,  the function $\log \abs{f}$ locally H\"older continuous in $\R^n$. In particular, $\log \abs{f}$ is locally bounded. However, if $f(x) = 0$ for some $x \in \R^n$, then $\log \abs{f(x)} = -\infty$. We conclude that $f$ cannot have any zeroes.
\end{proof}

\section{The Liouville theorem}

The remaining part of this paper  is devoted to proving the Liouville theorem formulated in Theorem \ref{thm:Liouville}.

We recall from the introduction that our approach is to consider a function ``$\log f$'' from $\R^n$ to $\R \times \S^{n-1}$. This mapping is well defined and satisfies a similar distortion inequality as the classical complex logarithm map. The differential inequality makes it possible to show that the weak derivative of our ``$\log f$'' lies in $L^{n-\eps}(\R^n)$ for some $\eps>0$. The argument for this goes back to two remarkable papers by Iwaniec and Martin~\cite{Iwaniec-Martin_Acta} (for even dimensions) and Iwaniec~\cite{Iwaniec_Annals} (for all dimensions), where they proved local integral estimates of quasiregular mappings below the natural exponent $n$. Later,  a  short  proof was given by  Faraco and Zhong~\cite{Faraco-Zhong_Caccioppoli}. We in turn perform a global version of the Lipschitz truncation argument of Faraco and Zhong in our setting.

\subsection{The logarithm with a manifold target}

Let $n \geq 2$. Then there exists a smooth mapping $s \colon \R \times \S^{n-1} \to \R^n \setminus \{0\}$, defined by
\[
	s(t, \theta) = e^t \theta
\]
for $t \in \R$ and $\theta \in \S^{n-1} \subset \R^n$. The map $s$ is conformal, with $\abs{Ds(t, \theta)}^n = J_s(t, \theta) = e^{nt}$. The inverse of $s$ is given by
\[
	s^{-1}(x) = \left(\log \abs{x}, \frac{x}{\abs{x}}\right)
\]
for $x \in \R^n \setminus \{0\}$. A simple calculation yields that $1 = (J_s \circ s^{-1}) J_{s^{-1}} = e^{n\log \abs{x}}J_{s^{-1}}$, and therefore
\[
	\abs{Ds^{-1}(x)}^n = J_{s^{-1}}(x) = \frac{1}{\abs{x}^n}.
\]

We use the inverse $s^{-1}$ to take a ``logarithm'' of our mapping $f$.

\begin{lemma}\label{lem:mfld_valued_log_props}
	Suppose that $f \in W^{1,n}_\loc(\R^n, \R^n)$ satisfies the heterogeneous distortion inequality \eqref{eq:main_definition} with $K \in [1, \infty)$ and $\sigma \in L^{n}(\R^n) \cap L^{n+\eps}_\loc(\R^n)$, for some $\eps > 0$. Suppose also that $f$ is bounded, and that $f$ is not the constant mapping $f \equiv 0$. Denote
	\[
		h = s^{-1}\circ f(x) = \left(\log \abs{f}, \frac{f}{\abs{f}}\right).
	\]
	Then $h$ has the following properties:
	\begin{enumerate}
		\item \label{enum:h_prop_sobolev} $h$ is continuous and $h \in W^{1,n}_\loc(\R^n, \R \times \S^{n-1})$;
		\item \label{enum:h_prop_globalint} we have $\abs{Dh} \in L^n(\R^n)$;
		\item \label{enum:h_prop_distortion} we have
		\[
			\abs{Dh(x)}^n \leq K J_h(x) + \sigma^n(x)
		\]
		for a.e.  $x\in \R^n$.
	\end{enumerate}
\end{lemma}
\begin{proof}
	By Theorem \ref{thm:Holder}, we have that $f$ is continuous, and by Theorem \ref{thm:nonzero}, we have that the image of $f$ does not meet zero. Hence, $h$ is well defined and continuous. We also easily see that $h \in W^{1,n}_\loc(\R^n, \R \times \S^{n-1})$, since if $B$ is a ball compactly contained in $\R^n \setminus \{0\}$, then $s\vert_{B}$ is a smooth bilipschitz chart. Hence, we have \eqref{enum:h_prop_sobolev}.
	
	We then note that since $s^{-1}$ is conformal, we have
	\[
		\abs{Dh} = \left(\abs{D s^{-1}} \circ f\right) \abs{Df} = \frac{\abs{Df}}{\abs{f}}.
	\]
	Hence, we have by Lemma \ref{lem:log_gradient_int} that $\abs{Dh} \in L^n(\R^n)$, proving \eqref{enum:h_prop_globalint}. Finally, we prove \eqref{enum:h_prop_distortion} by computing that
	\begin{multline*}
		\abs{Dh}^n
		= \left(\abs{D s^{-1}} \circ f\right) \abs{Df}
		\leq (J_{s^{-1}} \circ f) (KJ_f + \sigma^n \abs{f}^n)\\
		= K (J_{s^{-1}} \circ f) J_f + \frac{\sigma^n \abs{f}^n}{\abs{f}^n}
		= K J_h + \sigma^n.
	\end{multline*}
\end{proof}

\subsection{Integrability below the natural exponent}
According to Lemma~\ref{lem:mfld_valued_log_props}, the logarithmic mapping $h \colon \R^n \to  \R \times \S^{n-1}$ lies in $W^{1,n}_\loc(\R^n, \R \times \S^{n-1})$ and  it solves the distortion inequality,
\begin{equation}\label{eq:logdistineq}
\abs{Dh(x)}^n \le K J_h(x) +\sigma^n(x) \quad  \textnormal{ for } K\in [1, \infty) \textnormal{ and  a.e. } x\in \R^n \, .  
\end{equation}
Since $\abs{Dh}\in L^n (\R^n)$, the integral of the Jacobian $J_h$ over $\R^n$ vanishes. Therefore, the natural integral estimate for the logarithmic map over the entire space reads as follows
\[\int_{\R^n} \abs{Dh}^n \le \int_{\R^n} \sigma^n \, .   \]
The next lemma gives the key global integrability estimate for the differential below the natural exponent $n$.

\begin{lemma}\label{lem:lower_integrability}
	Suppose that a mapping $h \colon \R^n \to \R \times \S^{n-1}$ is continuous, and  that $h \in W^{1,n}_\loc(\R^n, \R \times \S^{n-1})$ with  $\abs{Dh} \in L^n(\R^n)$.  If $h$ satisfies the distortion inequality~\eqref{eq:logdistineq} with $\sigma \in L^{n-\eps}(\R^n) \cap L^{n}(\R^n)$ for some $\eps > 0$,  then there exists $\eps' = \eps'(n, K, \eps) \in (0, \eps)$ such that $\abs{Dh} \in L^{n-\eps'}(\R^n)$. In particular, we have the  estimate
	\[
		\int_{\R^n} \abs{Dh}^{n-\eps'}
		\leq C(\eps') \int_{\R^n} \sigma^{n-\eps'},
	\]
	where $C(\eps') \to 1$ as our choice of $\eps'$ tends to $0$.
\end{lemma}

\begin{proof}
	We may assume $\eps < 1$. We denote $h = (h_\R, h_{\S^{n-1}})$, where $h_\R \colon \R^n \to \R$ and $h_{\S^{n-1}} \colon \R^n \to \S^{n-1}$. Let
	\[
		g = \abs{Dh} + \sigma,
	\]
	and for every $\lambda > 0$, let
	\[
		F_\lambda = \{x \in \R^n : M(g)(x) \leq \lambda\}.
	\]
	
	Suppose that $x, y \in F_\lambda$. Then by a pointwise Sobolev estimate, we have for every $i \in \{1, \dots, n\}$ that
	\begin{multline*}
		\abs{h_{\R}(x) - h_{\R}(y)} \leq C_n \abs{x - y} (M(\abs{\nabla h_{\R}})(x) + M(\abs{\nabla h_{\R}})(y))\\
		\leq C_n \abs{x - y} (M(g)(x) + M(g)(y)) \leq C_n \lambda \abs{x - y}.
	\end{multline*}
	Hence, $h_{\R}$ is $C_n \lambda$-Lipschitz in $F_\lambda$. Consequently, by using the McShane extension theorem~\cite{McShane}, we find a $C_n \lambda$-Lipschitz map $h_{\R, \lambda} \colon \R^n \to \R$ such that $h_{\R, \lambda} \vert_{F_\lambda} = h_{\R} \vert_{F_\lambda}$. We denote $h_\lambda = (h_{\R, \lambda}, h_{\S^{n-1}})$.
	
	We point out that we have
	\[
		\abs{D h_\lambda} \leq (1+ C_n) M(g)
	\]
	a.e.\ in $\R^n$. Indeed, we have $\abs{D h_\lambda} = \abs{Dh} \leq g \leq M(g)$ a.e.\ in $F_\lambda$, and since $\abs{\nabla h_{\R, \lambda}} \leq C_n \lambda$, we also have $\abs{D h_\lambda} \leq \abs{Dh} + C_n \lambda \leq (1+C_n) M(g)$ a.e.\ in $\R^n \setminus F_\lambda$. Since $g \in L^n(\R^n)$, we also have $M(g) \in L^n(\R^n)$, and therefore $\abs{D h_\lambda} \in L^n(\R^n)$. Hence, we may apply the case of Lemma \ref{lem:zero_Jacobian} with a manifold target, obtaining that 
	\[
		\int_{\R^n} J_{h_\lambda} = 0.
	\]
	
	For $r > 0$, we denote $B_r = B^n(0, r)$. Since $J_h = J_{h_\lambda}$ in $F_\lambda$, we may therefore estimate that
	\begin{align*}
		\abs{\int_{B_r \cap F_\lambda} J_h}
		&= \abs{\int_{\R^n \setminus (B_r \cap F_\lambda)} J_{h_\lambda}}\\
		&\leq \abs{\int_{B_r \setminus F_\lambda} J_{h_\lambda}}
			+ \abs{\int_{\R^n \setminus B_r} J_{h_\lambda}}\\
		&\leq (1+C_n)^n \left( 
			\int_{B_r \setminus F_\lambda} \lambda 	M(g)^{n-1} 
			+ \int_{\R^n \setminus B_r} M(g)^n\right). 
	\end{align*}
	Moreover, since $\abs{Dh}^n \leq K J_h + \sigma^n$, we have
	\[
		\int_{B_r \cap F_\lambda} \abs{Dh}^n 
		\leq K\abs{\int_{B_r \cap F_\lambda} J_h}
		+ \int_{B_r \cap F_\lambda} \sigma^n.
	\]
	We now chain these estimates together, and multiply by $\lambda^{-1-\eps'}$, where $\eps' \in (0, \eps)$. We obtain
	\begin{multline}\label{eq:li_step1}
		\int_{B_r} \abs{Dh}^n \lambda^{-1-\eps'} \chi_{F_\lambda}
		\leq \int_{B_r} \sigma^n \lambda^{-1-\eps'} 
			\chi_{F_\lambda}\\
		+  (1 + C_n)^n K \left(\int_{B_r} \lambda^{-\eps'}M(g)^{n-1} 
			\chi_{\R^n \setminus F_\lambda}
			+ \lambda^{-1-\eps'} \int_{\R^n \setminus B_r} M(g)^n\right).
	\end{multline}
	
	Let $t > 0$. We now integrate \eqref{eq:li_step1} from $t$ to $\infty$ with respect to $\lambda$, and use the Fubini--Tonelli theorem to switch the order of integration. Observe that $\chi_{F_\lambda}(x) = 0$ if $\lambda < M(g)(x)$, and $\chi_{F_\lambda}(x) = 1$ otherwise. Hence,
	\begin{multline*}
		\int_{t}^\infty \lambda^{-1-\eps'} \chi_{F_\lambda}(x) \dd \lambda
		= \int_{\max(t, Mg(x))}^\infty \lambda^{-1-\eps'} \dd \lambda\\
		= \frac{\left[\max(t, Mg(x))\right]^{-\eps'}}{\eps'}
		\leq \frac{\left[Mg(x)\right]^{-\eps'}}{\eps'}
	\end{multline*}
	for a.e.\ $x \in \R^n$. Moreover, we also have
	\[
		\int_{t}^\infty \lambda^{-\eps'} 
			\chi_{\R^n \setminus F_\lambda}(x) \dd \lambda
		= \begin{cases}
		\displaystyle{\int_t^{M(g)(x)} \lambda^{-\eps'}} 
		& \text{if } Mg(x) > t\\
			0 \phantom{\displaystyle{\int}}
		& \text{if } Mg(x) \leq t
		\end{cases},
	\]
	and hence
	\begin{multline*}
		\int_{t}^\infty \lambda^{-\eps'} \chi_{\R^n \setminus F_\lambda}(x) \dd \lambda
		= \chi_{\R^n \setminus F_t}(x)  \int_{t}^{Mg(x)} \lambda^{-\eps'} \dd \lambda\\
		= \chi_{\R^n \setminus F_t}(x)\frac{\left[Mg(x)\right]^{1-\eps'} - t^{1-\eps'}}{1-\eps'} 
		\leq \chi_{\R^n \setminus F_t}(x) \frac{\left[Mg(x)\right]^{1-\eps'}}{1 - \eps'} 
	\end{multline*}
	for a.e.\ $x \in \R^n$. In conclusion, we obtain the estimate
	\begin{multline}\label{eq:li_step2}
		\int_{B_r} \abs{Dh}^n \left[\max(M(g), t)\right]^{-\eps'}
		\leq \int_{B_r} \sigma^n M(g)^{-\eps'}\\
		+ (1 + C_n)^n K \left( \frac{\eps'}{1 - \eps'} \int_{B_r \setminus F_t} M(g)^{n-\eps'}
		+ t^{-\eps'}\int_{\R^n \setminus B_r} M(g)^n \right).
	\end{multline}
	
	We then further estimate some of the terms in \eqref{eq:li_step2}. On the left hand side, we observe that if $x \notin F_t$, then $M(g)(x) > t$, and therefore $\max(M(g)(x), t) = M(g)(x)$. Hence, we have
	\[
		\int_{B_r} \abs{Dh}^n \left[\max(M(g), t)\right]^{-\eps'}
		\geq \int_{B_r \setminus F_t} \abs{Dh}^n M(g)^{-\eps'}.
	\]
	For the first term on the right hand side, we see from the definition of $g$ that $\sigma \leq g \leq M(g)$, and therefore obtain
	\[
		\int_{B_r} \sigma^n M(g)^{-\eps'}
		\leq \int_{B_r} \sigma^{n-\eps'}.
	\]
	For the remaining terms, we use the strong Hardy--Littlewood maximal inequality, where the same constant $M_n$ can be used for all exponents in the interval $[n-1, n]$. Moreover, we also estimate the third term by $g^{n-\eps'} \leq 2^{n} ( \abs{Dh}^{n-\eps'} + \sigma^{n-\eps'} ).$ After all these estimates of individual terms, we obtain a total estimate of the form
	\begin{multline}\label{eq:li_step3}
		\int_{B_r \setminus F_t} \abs{Dh}^n M(g)^{-\eps'}
		\leq \left( 1 + \frac{(2+2C_n)^n K \eps'}{1 - \eps'}\right) \int_{B_r} \sigma^{n-\eps'}\\
		+ \frac{(2+2C_n)^n K \eps'}{1 - \eps'} \int_{B_r \setminus F_t} \abs{Dh}^{n-\eps'}
		+ (2+2C_n)^n K t^{-\eps'} \int_{\R^n \setminus B_r} g^n.
	\end{multline}
	
	We then use Young's inequality to obtain the estimate
	\begin{align*}
		&\int_{B_r \setminus F_t} \abs{Dh}^{n-\eps'}\\
		&\qquad= \int_{B_r \setminus F_t} \Bigl( \abs{Dh}^{n-\eps'} 
			M(g)^{-\frac{\eps'(n-\eps')}{n}} \Bigr)
			\Bigl( M(g)^{\frac{\eps'(n-\eps')}{n}} \Bigr)\\
		&\qquad\leq \frac{(n-\eps')}{n}
			\int_{B_r \setminus F_t} \abs{Dh}^n M(g)^{-\eps'}
			+ \frac{\eps'}{n}
			\int_{B_r \setminus F_t} M(g)^{n-\eps'}\\
		&\qquad\leq \int_{B_r \setminus F_t} \abs{Dh}^n M(g)^{-\eps'}
			+ \eps' M_n 2^n \int_{B_r \setminus F_t} \left( \abs{Dh}^{n-\eps'}
				+ \sigma^{n-\eps'} \right).
	\end{align*}
	Hence, combining this with \eqref{eq:li_step3}, we now have
	\begin{multline}\label{eq:li_step4}
		\int_{B_r \setminus F_t} \abs{Dh}^{n-\eps'}
			\leq \left( 1 + \delta\right) \int_{B_r} 
				\sigma^{n-\eps'}
		+ \delta \int_{B_r \setminus F_t} \abs{Dh}^{n-\eps'}\\
		+ (2+2C_n)^n K t^{-\eps'} \int_{\R^n \setminus B_r} g^n,
	\end{multline}
	where $\delta = 2^n \eps'((1+C_n)^nK/(1-\eps') + M_n)$. We then select $\eps'$ small enough that $\delta < 1$. Since $B_r \setminus F_t$ is of finite measure, $\abs{Dh}^{n-\eps'}$ is integrable over it, and we may absorb its term from the right hand side of \eqref{eq:li_step4} to the left hand side. We obtain the estimate
	\[
		\int_{B_r \setminus F_t} \abs{Dh}^{n-\eps'}
		\leq \frac{1 + \delta}{1 - \delta} \int_{B_r} \sigma^{n-\eps'}
		+ \frac{(2nC_n)^n K t^{-\eps'}}{1-\delta} \int_{\R^n \setminus B_r} g^n.
	\]
	We let $r \to \infty$. Since $g \in L^n(\R^n)$, this makes the final term vanish, yielding
	\[
		\int_{\R^n \setminus F_t} \abs{Dh}^{n-\eps'}
		\leq \frac{1 + \delta}{1 - \delta} \int_{\R^n}  \sigma^{n-\eps'} < \infty.
	\]
	Note that we may assume that $\bigcup_{t > 0} \R^n \setminus F_t = \R^n$. Indeed, otherwise $M(g)$ has a zero; this is possible only if $g \equiv 0$, in which case $h$ is constant and the claim is trivial. Hence, by letting $t \to 0^+$, the claim follows.
\end{proof}

\subsection{Proof of the Liouville theorem}

It remains to complete the proof of Theorem \ref{thm:Liouville}. We recap the statement before the proof.

\begin{customthm}{\ref{thm:Liouville}}
	Suppose that $f \in W^{1,n}_\loc(\R^n, \R^n)$ satisfies the heterogeneous distortion inequality \eqref{eq:main_definition} with $K \in [1, \infty)$ and $\sigma \in L^{n-\eps}(\R^n) \cap L^{n+\eps}(\R^n)$, for some $\eps > 0$. If $f$ is bounded and $\lim_{x \to \infty} \abs{f(x)} = 0$, then $f \equiv 0$.
\end{customthm}

\begin{proof}
	Suppose towards to a contradiction that $f$ is bounded and  $\lim\limits_{x \to \infty} \abs{f(x)} = 0$, but  $f$ is not identically zero. By Theorems \ref{thm:Holder} and \ref{thm:nonzero}, we have that $f$ is continuous and has no zeros. Hence, we may define the ``logarithmic''  mapping $h \colon \R^n \to \R \times \S^{n-1}$ by
	\[
		h(x) = \left( \log \abs{f}, \frac{\abs{f}}{\abs{f}} \right). 
	\]
	By Lemma \ref{lem:mfld_valued_log_props}, we have that $h \in W^{1,n}_\loc(\R^n, \R \times \S^{n-1})$, $\abs{Dh} \in L^n(\R^n)$, and $\abs{Dh}^n \leq K J_h + \sigma^n$. Combining this with  Lemma \ref{lem:lower_integrability} we conclude that $\abs{Dh} \in L^{n-\eps'}(\R^n)$ for some $\eps' > 0$. In particular, since $\abs{\nabla \log \abs{f}} \leq \abs{Dh}$, we have 
	\begin{equation}\label{eq:logintegrability}
	\int_{\R^n} \abs{\nabla \log \abs{f}}^{n-\eps'}    \le  \int_{\R^n}  \abs{Dh}^{n-\eps'}  \le C(\eps') \int_{\R^n} \sigma^{n-\eps'} \, . 
	\end{equation}
	
	Consider now balls of the form $B_i = B^n(0, 2^i)$. Our goal is to show that the integral average of $\abs{\log \abs{f}}$ over $B_i$, denoted by $ (\log \abs{f})_{B_{i}} $, is  bounded independently of $i\in \mathbb N \cup \{0\}$. By the Sobolev-Poincar\'e inequality~\cite[4.5.2]{Evans-Gariepy-book} and~\eqref{eq:logintegrability}  we have
	\begin{multline*}
		\abs{(\log \abs{f})_{B_{i-1}} - (\log \abs{f})_{B_{i}}}
		\leq \frac{2^n}{m_n(B_i)}\int_{B_i} 
			\abs{\log \abs{f} - (\log \abs{f})_{B_{i}}}\\
		\leq C_n 2^n 2^{i} \left( \frac{1}{m_n(B_i)} \int_{B_i}
			\abs{\nabla \log \abs{f}}^{n-\eps'} \right)^\frac{1}{n-\eps'}\\
		\leq C_n 2^{n} 2^{i - \frac{ni}{n-\eps'}}\omega_n^{-\frac{1}{n-\eps'}}
			\left( \int_{\R^n}
			\abs{\nabla \log \abs{f}}^{n-\eps'} \right)^\frac{1}{n-\eps'}\\
		\leq C_{n'} C(\eps') 2^n \max(1, \omega_n^{-n}) 2^{-\frac{\eps'}{n-\eps'}i}
			\left( \int_{\R^n} \sigma^{n-\eps'} \right)^\frac{1}{n-\eps'}.
	\end{multline*}
	Consequently, we have that
	\begin{multline*}
		\abs{(\log \abs{f})_{B_{i}} - (\log \abs{f})_{B_{0}}}\\
		\leq C_{n} C(\eps') 2^n \max(1, \omega_n^{-n}) 
		\left( \int_{\R^n} \sigma^{n-\eps'} \right)^\frac{1}{n-\eps'}
		\sum_{i=0}^\infty 2^{-\frac{\eps'}{n-\eps'}i}
		< \infty.
	\end{multline*}
	Since $\log \abs{f} \in L^1_\loc$ by Lemma \ref{lem:log_f_is_Sobolev}, we have $\abs{(\log \abs{f})_{B_{0}}} < \infty$. However, since $\lim\limits_{x \to \infty} f(x) = 0$, we have $\lim_{i \to \infty} (\log \abs{f})_{B_{i}} = -\infty$. This leads to a contradiction, and the claim follows.
\end{proof}


\bibliographystyle{abbrv}
\bibliography{sources}

\begin{thebibliography}{10}

\bibitem{Astala-Iwaniec-Martin_Book}
K.~Astala, T.~Iwaniec, and G.~Martin.
\newblock {\em Elliptic partial differential equations and quasiconformal
  mappings in the plane}.
\newblock Princeton university press, 2009.

\bibitem{Astala-Paivarinta}
K.~Astala and L.~P\"{a}iv\"{a}rinta.
\newblock Calder\'{o}n's inverse conductivity problem in the plane.
\newblock {\em Ann. of Math. (2)}, 163(1):265--299, 2006.

\bibitem{Caccioppoli}
R.~Caccioppoli.
\newblock Limitazioni integrali per le soluzioni di un'equazione lineare
  ellitica a derivate parziali.
\newblock {\em Giorn. Mat. Battaglini (4)}, 4(80):186--212, 1951.

\bibitem{Campanato-space}
S.~Campanato.
\newblock Equazioni ellittiche del {${\rm II}\deg $} ordine espazi {${
  L}^{(2,\lambda )}$}.
\newblock {\em Ann. Mat. Pura Appl. (4)}, 69:321--381, 1965.

\bibitem{Cauchy}
A.~Cauchy.
\newblock {\em C.R. Acad. Sci. Paris}, (19):1377--1384, (1844).

\bibitem{Convent-VanSchaftingen_Sobolev}
A.~Convent and J.~Van~Schaftingen.
\newblock Intrinsic co-local weak derivatives and {S}obolev spaces between
  manifolds.
\newblock {\em Ann. Sc. Norm. Super. Pisa Cl. Sci. (5)}, 16(1):97--128, 2016.

\bibitem{Evans-Gariepy-book}
L.~C. Evans and R.~F. Gariepy.
\newblock {\em Measure theory and fine properties of functions}.
\newblock Textbooks in Mathematics. CRC Press, Boca Raton, FL, revised edition,
  2015.

\bibitem{Faraco-Zhong_Caccioppoli}
D.~Faraco and X.~Zhong.
\newblock A short proof of the self-imroving regularity of quasiregular
  mappings.
\newblock {\em Proc. Amer. Math. Soc.}, 134(1):187--192, 2005.

\bibitem{Fonseca-Gangbo-book}
I.~Fonseca and W.~Gangbo.
\newblock {\em Degree theory in analysis and applications}, volume~2 of {\em
  Oxford Lecture Series in Mathematics and its Applications}.
\newblock The Clarendon Press, Oxford University Press, New York, 1995.
\newblock Oxford Science Publications.

\bibitem{Gehring}
F.~W. Gehring.
\newblock The {$L^{p}$}-integrability of the partial derivatives of a
  quasiconformal mapping.
\newblock {\em Acta Math.}, 130:265--277, 1973.

\bibitem{Hajlasz-Iwaniec-Maly-Onninen}
P.~Haj{\l}asz, T.~Iwaniec, J.~Mal{\`y}, and J.~Onninen.
\newblock Weakly differentiable mappings between manifolds.
\newblock {\em Mem. Amer. Math. Soc.}, 192(899), 2008.

\bibitem{Hedberg}
L.~I. Hedberg.
\newblock On certain convolution inequalities.
\newblock {\em Proc. Amer. Math. Soc.}, 36:505--510, 1972.

\bibitem{Heinonen-Kilpelainen-Martio_book}
J.~Heinonen, T.~Kilpel{\"a}inen, and O.~Martio.
\newblock {\em Nonlinear potential theory of degenerate elliptic equations}.
\newblock Dover, 2006.

\bibitem{Hencl-Koskela-book}
S.~Hencl and P.~Koskela.
\newblock {\em {Lectures on mappings of finite distortion}}, volume 2096 of
  {\em {Lecture Notes in Mathematics}}.
\newblock Springer, Cham, 2014.

\bibitem{Hencl-Rajala}
S.~Hencl and K.~Rajala.
\newblock Optimal assumptions for discreteness.
\newblock {\em Arch. Ration. Mech. Anal.}, 207(3):775--783, 2013.

\bibitem{Iwaniec_Annals}
T.~Iwaniec.
\newblock $p$-harmonic tensors and quasiregular mappings.
\newblock {\em Ann. of Math.}, 136(3):589--624, 1992.

\bibitem{Iwaniec-Koskela-Onninen-Invent}
T.~Iwaniec, P.~Koskela, and J.~Onninen.
\newblock Mappings of finite distortion: monotonicity and continuity.
\newblock {\em Invent. Math.}, 144(3):507--531, 2001.

\bibitem{Iwaniec-Martin_Acta}
T.~Iwaniec and G.~Martin.
\newblock Quasiregular mappings in even dimensions.
\newblock {\em Acta Math.}, 170(1):29--81, 1993.

\bibitem{Iwaniec-Martin_book}
T.~Iwaniec and G.~Martin.
\newblock {\em Geometric function theory and non-linear analysis}.
\newblock Clarendon Press, 2001.

\bibitem{Iwaniec-Scott-Stroffolini}
T.~Iwaniec, C.~Scott, and B.~Stroffolini.
\newblock Nonlinear {H}odge theory on manifolds with boundary.
\newblock {\em Ann. Mat. Pura. Appl. (4)}, 177(1):37--115, 1999.

\bibitem{Iwaniec-Sverak}
T.~Iwaniec and V.~\v{S}ver\'{a}k.
\newblock On mappings with integrable dilatation.
\newblock {\em Proc. Amer. Math. Soc.}, 118(1):181--188, 1993.

\bibitem{holomorphicmotion}
R.~Ma\~{n}\'{e}, P.~Sad, and D.~Sullivan.
\newblock On the dynamics of rational maps.
\newblock {\em Ann. Sci. \'{E}cole Norm. Sup. (4)}, 16(2):193--217, 1983.

\bibitem{McShane}
E.~J. McShane.
\newblock Extension of range of functions.
\newblock {\em Bull. Amer. Math. Soc.}, 40(12):837--842, 1934.

\bibitem{Morrey-Holder}
C.~B. Morrey, Jr.
\newblock On the solutions of quasi-linear elliptic partial differential
  equations.
\newblock {\em Trans. Amer. Math. Soc.}, 43(1):126--166, 1938.

\bibitem{Morrey-space}
C.~B. Morrey, Jr.
\newblock Second-order elliptic systems of differential equations.
\newblock In {\em Contributions to the theory of partial differential
  equations}, Annals of Mathematics Studies, no. 33, pages 101--159. Princeton
  University Press, Princeton, N. J., 1954.

\bibitem{Morrey-Book}
C.~B. Morrey, Jr.
\newblock {\em Multiple integrals in the calculus of variations}.
\newblock Die Grundlehren der mathematischen Wissenschaften, Band 130.
  Springer-Verlag New York, Inc., New York, 1966.

\bibitem{Onninen-Zhong_MFD-loglog-proof}
J.~Onninen and X.~Zhong.
\newblock Mappings of finite distortion: a new proof for discreteness and
  openness.
\newblock {\em Proc. Roy. Soc. Edinburgh Sect. A}, 138(5), 2008.

\bibitem{Reshetnyak_continuity}
Y.~G. Reshetnyak.
\newblock Bounds on moduli of continuity for certain mappings.
\newblock {\em Sibirsk. Mat. Zh.}, 7:1106--1114, 1966.
\newblock (Russian).

\bibitem{Reshetnyak_Liouville}
Y.~G. Reshetnyak.
\newblock The {L}iouville theorem with mininal regularity conditions.
\newblock {\em Sibirsk. Mat. Zh.}, 8:835--840, 1967.
\newblock (Russian).

\bibitem{Reshetnyak_Theorem2}
Y.~G. Reshetnyak.
\newblock On the condition of the boundedness of index for mappings with
  bounded distortion.
\newblock {\em Sibirsk. Mat. Zh.}, 9:368--374, 1967.
\newblock (Russian).

\bibitem{Reshetnyak_QROrigin}
Y.~G. Reshetnyak.
\newblock Space mappings with bounded distortion.
\newblock {\em Sibirsk. Mat. Zh.}, 8:629--659, 1967.
\newblock (Russian).

\bibitem{Reshetnyak-book}
Y.~G. Reshetnyak.
\newblock {\em {Space mappings with bounded distortion}}, volume~73 of {\em
  {Translations of Mathematical Monographs}}.
\newblock American Mathematical Society, Providence, RI, 1989.

\bibitem{Rickman_book}
S.~Rickman.
\newblock {\em Quasiregular mappings}, volume~26.
\newblock Springer-Verlag, 1993.

\bibitem{Manfredi-Villamor}
E.~Villamor and J.~J. Manfredi.
\newblock An extension of {R}eshetnyak's theorem.
\newblock {\em Indiana Univ. Math. J.}, 47(3):1131--1145, 1998.

\bibitem{Vodopanov-Goldstein-cont}
S.~K. Vodop'janov and V.~M. Gol'd\v{s}te\u{\i}n.
\newblock Quasiconformal mappings, and spaces of functions with first
  generalized derivatives.
\newblock {\em Sibirsk. Mat. \v{Z}.}, 17(3):515--531, 715, 1976.

\bibitem{Ziemer-book}
W.~P. Ziemer.
\newblock {\em Weakly differentiable functions}, volume 120 of {\em Graduate
  Texts in Mathematics}.
\newblock Springer-Verlag, New York, 1989.
\newblock Sobolev spaces and functions of bounded variation.

\bibitem{Zorich}
V.~A. Zori\v{c}.
\newblock M. {A}. {L}avrent'ev's theorem on quasiconformal space maps.
\newblock {\em Mat. Sb. (N.S.)}, 74 (116):417--433, 1967.

\end{thebibliography}

\end{document}